\documentclass[10pt]{article}
\usepackage[all,2cell]{xy} \UseAllTwocells \SilentMatrices
\usepackage{latexsym,amsfonts,amssymb}
\usepackage{amsmath,amsthm,amscd}
\usepackage{hyperref,psfrag}
\usepackage{color}
\usepackage{etoolbox}
\usepackage[dvips]{epsfig}
\usepackage{psfrag}
\usepackage{pinlabel}

\usepackage{graphicx}
\usepackage{tikz-cd}

\usepackage{a4wide}
\usepackage{geometry}

\usepackage{epigraph,wrapfig}

\normalfont\upshape

\usepackage{fancyhdr}
\theoremstyle{definition}
\newtheorem{thm}{Theorem}[section]
\newtheorem{cor}[thm]{Corollary}
\newtheorem{conj}[thm]{Conjecture}
\newtheorem{lem}[thm]{Lemma}
\newtheorem{rem}[thm]{Remark}
\newtheorem{prop}[thm]{Proposition}
\newtheorem{defn}[thm]{Definition}
\newtheorem{example}[thm]{Example}

\numberwithin{equation}{section}

\usepackage{bbm}
\def\ot{\otimes}

\def\Z{{\mathbbm Z}}

\def\F{{\mathbbm F}}

\def\1{{\mathbbm{1}}}

\newcommand{\Hom}{{\rm Hom}}

\renewcommand{\to}{\rightarrow}

\newcommand{\co}{\colon}

\newcommand{\id}{{\rm id}}

\newcommand{\pa}{\partial}

\newcommand{\End}{{\rm End}}

\def\shuffle{\,\raise 1pt\hbox{$\scriptscriptstyle\cup{\mskip
               -4mu}\cup$}\,}

\newcommand{\refequal}[1]{\xy {\ar@{=}^{#1}
(-1,0)*{};(1,0)*{}};
\endxy}

\def\Uc{\mathcal{U}}
\def\Pc{\mathcal{P}}
\def\OM{\mathbb{O}}

\def\om{\varpi}

\newcommand{\WC}{\mathcal{W}}

\newcommand{\barg}{\bar{g}}
\newcommand{\bark}{\bar{\kappa}}
\newcommand{\bard}{\bar{d}}

\newcommand{\ig}[2]{\vcenter{\xy (0,0)*{\includegraphics[scale=#1]{fig/#2}} \endxy}}

\newcommand{\lineblue}{\ig{.5}{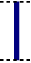}}

\newcommand{\linegreen}{\ig{.5}{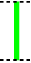}}
\newcommand{\startdotblue}{\ig{.5}{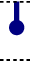}}

\newcommand{\finaldotblue}{\ig{.5}{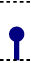}}

\newcommand{\finaldotgreen}{\ig{.5}{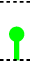}}
\newcommand{\barbblue}{\ig{.5}{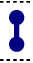}}

\newcommand{\brokenblue}{\ig{.5}{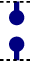}}

\newcommand{\splitblue}{\ig{.5}{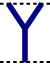}}

\newcommand{\mergeblue}{\ig{.5}{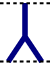}}

\newcommand{\capblue}{\ig{.5}{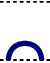}}
\newcommand{\cupblue}{\ig{.5}{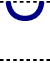}}
\newcommand{\Xbg}{\ig{.5}{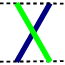}}
\newcommand{\Xgb}{\ig{.5}{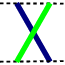}}
\newcommand{\sixvalent}{\ig{.5}{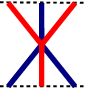}}
\newcommand{\brokeXII}{\ig{.5}{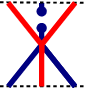}}
\newcommand{\brokeII}{\ig{.5}{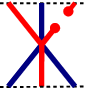}}
\newcommand{\brokeIV}{\ig{.5}{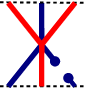}}
\newcommand{\brokeVI}{\ig{.5}{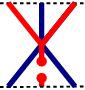}}
\newcommand{\brokeVIII}{\ig{.5}{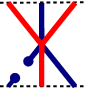}}
\newcommand{\brokeX}{\ig{.5}{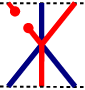}}
\newcommand{\brokeboth}{\ig{.5}{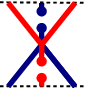}}
\newcommand{\cupcapthing}{\ig{.5}{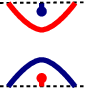}}

\newcommand{\poly}[1]{{
\labellist
\small\hair 2pt
 \pinlabel {$#1$} [ ] at 7 15
\endlabellist
\centering
\ig{.5}{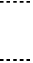}
}}
\newcommand{\longpoly}[1]{
\ig{.5}{space} \ig{.5}{space} {
\labellist
\small\hair 2pt
 \pinlabel {$#1$} [ ] at 7 15
\endlabellist
\centering
\ig{.5}{space}
} \ig{.5}{space} \ig{.5}{space} }
\newcommand{\tallpoly}[1]{{
\labellist
\small\hair 2pt
 \pinlabel {$#1$} [ ] at 20 20
\endlabellist
\centering
\ig{.5}{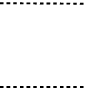}
}}
\newcommand{\otherpoly}[1]{{
\labellist
\small\hair 2pt
 \pinlabel {$#1$} [ ] at 7 15
\endlabellist
\centering
\ig{1}{space}
}}
\newcommand{\splitpoly}[3]{{
\labellist
\tiny\hair 2pt
 \pinlabel {$#1$} [ ] at 5 14
 \pinlabel {$#2$} [ ] at 33 14
 \pinlabel {$#3$} [ ] at 19 35
\endlabellist
\centering
\ig{1}{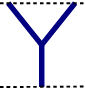}
}}
\newcommand{\mergepoly}[3]{{
\labellist
\tiny\hair 2pt
 \pinlabel {$#1$} [ ] at 5 30
 \pinlabel {$#2$} [ ] at 33 30
 \pinlabel {$#3$} [ ] at 19 8
\endlabellist
\centering
\ig{1}{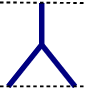}
}}
\newcommand{\Xbgpoly}[4]{{
\labellist
\tiny\hair 2pt
 \pinlabel {$#4$} [ ] at 2 21
 \pinlabel {$#2$} [ ] at 38 21
 \pinlabel {$#1$} [ ] at 20 35
 \pinlabel {$#3$} [ ] at 20 8
\endlabellist
\centering
\ig{1}{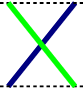}
}}
\newcommand{\Xgbpoly}[4]{{
\labellist
\tiny\hair 2pt
 \pinlabel {$#4$} [ ] at 2 21
 \pinlabel {$#2$} [ ] at 38 21
 \pinlabel {$#1$} [ ] at 20 35
 \pinlabel {$#3$} [ ] at 20 8
\endlabellist
\centering
\ig{1}{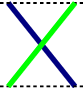}
}}
\newcommand{\needlepoly}[1]{{
\labellist
\tiny\hair 2pt
 \pinlabel {$#1$} [ ] at 20 20
\endlabellist
\centering
\ig{1}{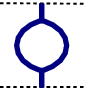}
}}
\newcommand{\linebluelabel}[1]{{\color{olive}{
\labellist
\small\hair 2pt
 \pinlabel {$#1$} [ ] at 7 15
\endlabellist
\centering
\ig{1}{lineblue}
}}}
\newcommand{\sixvalentlabel}[8]{{\color{olive}{
\labellist
\small\hair 2pt
 \pinlabel {$#1$} [ ] at 32 41
 \pinlabel {$#2$} [ ] at 40 38
 \pinlabel {$#3$} [ ] at 40 12
 \pinlabel {$#4$} [ ] at 32 9
 \pinlabel {$#5$} [ ] at 22 12
 \pinlabel {$#6$} [ ] at 22 38
 \pinlabel {$#7$} [ ] at 55 24
 \pinlabel {$#8$} [ ] at 7 24
\endlabellist
\centering
\ig{1}{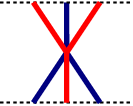}
}}}
\newcommand{\othersixvalentlabel}[8]{{\color{olive}{
\labellist
\small\hair 2pt
 \pinlabel {$#1$} [ ] at 32 41
 \pinlabel {$#2$} [ ] at 40 38
 \pinlabel {$#3$} [ ] at 40 12
 \pinlabel {$#4$} [ ] at 32 9
 \pinlabel {$#5$} [ ] at 22 12
 \pinlabel {$#6$} [ ] at 22 38
 \pinlabel {$#7$} [ ] at 55 24
 \pinlabel {$#8$} [ ] at 7 24
\endlabellist
\centering
\ig{1}{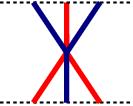}
}}}
\newcommand{\mergelabel}[5]{{\color{olive}{
\labellist
\small\hair 2pt
 \pinlabel {$#1$} [ ] at 32 41
 \pinlabel {$#2$} [ ] at 40 12
 \pinlabel {$#3$} [ ] at 22 12
 \pinlabel {$#4$} [ ] at 55 24
 \pinlabel {$#5$} [ ] at 7 24
\endlabellist
\centering
\ig{1}{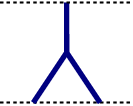}
}}}

\title{Categorifying Hecke algebras at prime roots of unity, part I}
\author{Ben Elias, You Qi}

\date{\today}

\begin{document}
%

\maketitle

\begin{abstract} We equip the type $A$ diagrammatic Hecke category with a special derivation, so that after specialization to characteristic $p$ it becomes a $p$-dg category. We
prove that the defining relations of the Hecke algebra are satisfied in the $p$-dg Grothendieck group. We conjecture that the $p$-dg Grothendieck group is isomorphic to the
Iwahori-Hecke algebra, equipping it with a basis which may differ from both the Kazhdan-Lusztig basis and the $p$-canonical basis. More precise conjectures will be found in the
sequel.

Here are some other results contained in this paper. We provide an incomplete proof of the classification of all degree $+2$ derivations on the diagrammatic Hecke category, and a
complete proof of the classification of those derivations for which the defining relations of the Hecke algebra are satisfied in the $p$-dg Grothendieck group. In particular, our
special derivation is unique up to duality and equivalence. We prove that no such derivation exists in simply-laced types outside of finite and affine type $A$. We also examine a particular Bott-Samelson bimodule in type $A_7$, which is indecomposable in characteristic $2$ but decomposable in all other characteristics. We prove that this Bott-Samelson bimodule admits no nontrivial fantastic filtrations in any characteristic, which is the analogue in the $p$-dg setting of being indecomposable. \end{abstract}

\setcounter{tocdepth}{2} \tableofcontents

\section{Introduction}

In order to understand what we do in this paper and why, it helps to know a bit about categorification at a root of unity, and its most popular techniques.

\subsection{Categorification at a root of unity}

\begin{defn} 
Let $A$ be an algebra over a ground ring $\Bbbk$. A $\Bbbk$-linear map $d \co A \to A$ is called an \emph{even differential}\footnote{A differential graded algebra has a degree one map called the differential which
satisfies the super Leibniz rule. For an even degree map, there is no difference between the ordinary and the super Leibniz rule.} or a \emph{derivation} if it satisfies the Leibniz rule
\begin{equation} 
d(fg) = d(f)g + f d(g) 
\end{equation}
for any $f,g \in A$.
\end{defn}

In his seminal paper \cite{KhoHopf}, Khovanov began the program of ``categorification at a root of unity.'' He defined a \emph{$p$-dg algebra} to be a graded algebra $A$ over a
field $\Bbbk$ of characteristic $p$, equipped with an even differential $d$ which is homogeneous of degree $2$, satisfying $d^p = 0$. A $p$-dg category is defined similarly. Khovanov defined the derived category of a $p$-dg category, and observed that its Grothendieck group is naturally a
module over the $p$-th cyclotomic integers. More precisely, it is a module over $\OM_p$, which is the extension of $\Z$ by a variable $\zeta$ for which $\zeta^2$ is a primitive
$p$-th root of unity. To categorify something like the quantum group at a root of unity, one should hunt for an interesting $p$-dg category.

The work of Khovanov-Qi \cite{KQ} successfully applied this idea to the categorification of quantum groups. They took the quiver Hecke algebra or KLR algebra
\cite{KhoLau11,Rouq2KM-pp} in simply-laced type, which was known to categorify the positive half of the quantum group at generic $q$, and equipped it with a degree $2$ derivation
$d$. After specialization to characteristic $p$, it becomes a $p$-dg category. They conjectured \cite[Conjecture 4.18]{KQ} that this $p$-dg quiver Hecke algebra categorifies the
positive half of the small quantum group at $q = \zeta$. They were able to prove this conjecture for $\mathfrak{sl}_2$, and to prove in general that the defining relations of the
small quantum group hold in the Grothendieck group \cite[Theorem 3.35, Theorem 4.14]{KQ}. The conjecture has been proven by Andrew Stephens for small quantum $\mathfrak{sl}_3$
\cite{StephensSL3}.

Before continuing, let us note a common theme to the construction of $p$-dg algebras: the differential is independent of the prime $p$. Typically there is just one differential defined on an integral form of the algebra, which satisfies $d^p = 0$ only after specialization to characteristic $p$. Here is why.

\begin{defn} An \emph{even differential graded algebra}, or simply \emph{edg-algebra}, $(A,d)$ is a graded $\Z$-algebra $A$ equipped with an even differential of degree $2$. An edg-algebra is called a \emph{gaea} (short for $\mathbb{G}_a$-equivariant algebra, or a metaphor for a ``global'' $p$-dg algebra) if the operator $d^{(k)} := \frac{d^k}{k!}$ can be defined over $\Z$. \end{defn}

A gaea specializes to a $p$-dg algebra after changing base from $\Z$ to any field of characteristic $p$, because \[d^p = p! d^{(p)} = 0.\]

\begin{example} \label{ex:standarddiffpoly} Let $R = \Z[x_1, \ldots, x_n]$ be a polynomial ring, graded such that $\deg(x_i) = 2$, and consider the operator 
\begin{equation} d = \sum_{i=1}^n x_i^2 \frac{\pa}{\pa_{x_i}}. \end{equation}
This is the even differential determined by the Leibniz rule and the equation
\begin{equation} d(x_i) = x_i^2 \end{equation}
for all $i$. We call this the \emph{standard differential} on the polynomial ring. The operator $d^{(k)}$ sends $x_i^\ell$ to $\binom{\ell+k-1}{k} x_i^{\ell+k}$, and is well-defined over $\Z$. One can prove that (after specialization to characteristic $p$) this $p$-dg algebra is quasi-isomorphic to the ground field $\Bbbk$, so that its $p$-dg Grothendieck group is just $\OM_p$. \end{example}

In this new terminology, Khovanov-Qi equipped the quiver Hecke algebra with the structure of a gaea. This illustrates the principle that, if you already have an additive categorification of an algebra
at generic $q$, you should try to equip it with a differential (extending to a gaea) in order to categorify the specialization to any prime root of unity. As further illustration
of this principle, we categorified the entire quantum group of $\mathfrak{sl}_2$ in \cite{EQpDGsmall,EQpDGbig}, equipping Lauda's category $\Uc$ and Khovanov-Lauda-Mackaay-Stosic's
category $\dot{\Uc}$ (see \cite{LauSL2} and \cite{KLMS} respectively) with such a differential.

In this paper, we equip the diagrammatic Hecke category in finite and affine type $A$ with the structure of a gaea, hoping to categorify the Iwahori-Hecke algebra at a root of
unity. We also prove a negative result in any other simply-laced type.

Let us note that the differential $d$ on the KLR algebra in type $A_1$ is, for all effective purposes, induced by the standard differential on the
polynomial ring discussed in Example \ref{ex:standarddiffpoly}. What is more interesting is that there is a large family of possible differentials on the KLR algebra, but that
only the standard differential (and its dual) gives rise to a $p$-dg algebra with the correct $p$-dg Grothendieck group! The same seems to be true of the diagrammatic Hecke
category.

To reiterate, constructing a differential on the Hecke category is not the hard part, and it is only the first step. The real difficulty lies in computing the $p$-dg Grothendieck
group. Let us explain some of the technology involved in understanding $p$-dg Grothendieck groups.

\subsection{Fantastic filtrations}

In an additive category, suppose we have a direct sum decomposition
\[ X \cong \bigoplus_{j \in I} M_j \] for some finite index set $I$. This produces a relation on the (split) Grothendieck group 
\begin{equation} \label{desired} [X] = \sum [M_j]. \end{equation}
If \eqref{desired} is a relation you want then you should prove it by constructing a direct sum decomposition. In practical terms, to prove that $X \cong \bigoplus M_j$, one must construct projection maps
\[ p_j \co X \to M_j \]
and inclusion maps
\[ i_j \co M_j \to X \]
which satisfy the basic axioms
\begin{subequations} \label{directsumdecomp}
\begin{equation} \label{pMiMidM} p_j i_j = \id_{M_j}, \end{equation}
\begin{equation} p_j i_k = 0 \quad \text{ if } j \ne k, \end{equation}
\begin{equation} \id_X = \sum_{j \in I} i_j p_j. \end{equation}
\end{subequations}

Now consider the exact same scenario, but in a $p$-dg category. To make a long story short, \eqref{directsumdecomp} is not enough information to deduce that \eqref{desired} holds in
the $p$-dg Grothendieck group. The object $X$ is not actually a direct sum of the objects $M_j$ because the differential need not preserve the summands. Let us say this more precisely.
The functor $\Hom(X,-)$ is (by definition) a representable left module over the category, where each piece $\Hom(X,Y)$ is equipped with a differential. Let $e_j = i_j p_j \in \End(X)$.
Then the direct summand $\Hom(X,-) e_j \subset \Hom(X,-)$ need not be preserved by the differential on Hom spaces. It is not even desirable for $\Hom(X,-) e_j$ to be preserved by the differential, as this condition often fails in practice.

However, one might hope that the differential acts on the pieces of this decomposition in an ``upper-triangular'' fashion, which would equip $X$ with a filtration whose
subquotients might agree with the $M_j$. This idea is codified in \cite[\S 5]{EQpDGsmall} in the notion of an \emph{Fc-filtration}. This is shorthand for either a \emph{fantastic filtration} or a \emph{finite cell filtration}, your choice. A direct sum decomposition as in \eqref{directsumdecomp} is said to \emph{lift to a Fc-filtration} if there exists a partial order on the
index set $I$ satisfying \begin{equation} \label{fcfilt} p_j d(i_k) = 0 \quad \text{ whenever } j \le k. \end{equation} If this can be done, then \eqref{desired} still holds in the
$p$-dg Grothendieck group.

The reader new to this theory should think of \eqref{fcfilt} as consisting of two separate statements. The first statement, $p_j d(i_k) = 0$ whenever $j < k$, implies that $\Hom(X,-)$ is filtered with subquotients $\Hom(X,-) e_j$. The second statement, $p_j d(i_k) = 0$ when $j = k$, implies that $\Hom(X,-) e_j$ and $\Hom(M_j,-)$ are isomorphic as $p$-dg modules (i.e. the natural isomorphism between these functors intertwines the differential). Since $\Hom(M_j,-)$ is representable, this implies that $\Hom(X,-) e_j$ is \emph{cofibrant} and \emph{compact}, which is needed for it to have a symbol in the Grothendieck group in the first place.

\begin{rem} A simpler idea is that of a \emph{dg-filtration} on an object $X$, which is a complete collection of orthogonal idempotents $e_j$ satisfying $e_j d(e_k) = 0$ for $j <
k$. This equips $\Hom(X,-)$ with a filtration by $p$-dg submodules which are additive summands. The additional data in an Fc-filtration goes one step further and proves that the subquotients
$\Hom(X,-) e_j$ are isomorphic to some known representable modules $\Hom(M_j,-)$. For a dg-filtration, it is not obvious that the subquotients $\Hom(X,-)e_j$ will be cofibrant. \end{rem}

It is not really important which partial order on $I$ gives rise to \eqref{fcfilt}, only that some partial order should exist. Practically speaking, the method for determining if a
direct sum decomposition lifts to a Fc-filtration is as follows. Construct an oriented graph $\Gamma_{I,d}$ with vertex set $I$, having an edge from $k$ to $j$ labeled by the degree
$+2$ morphism $p_j d(i_k)$. Erase all edges with the zero label. If $\Gamma_{I,d}$ has no oriented cycles (and in particular, no loops) then it is possible to find a partial order on
$I$ satisfying \eqref{fcfilt}, and consequently \eqref{desired} holds.

Direct sum decompositions are somewhat fluid: there are usually many valid choices for the projection maps $p=\{p_j\}$ and inclusion maps $i = \{i_j\}$. Note that the graph under study
depends on both the differential $d$ and the choice of projection and inclusion maps, so we could denote it by $\Gamma_{I,d,p,i}$ to emphasize this point. What is rather amazing is
that the Fc-filtration requirement is incredibly good at rigidifying the situation: while there may be large families of differentials $d$ on a category, and large families of
projection and inclusion maps, there is often a unique (up to symmetry) triple $(d,p,i)$ such that $\Gamma_{I,d,p,i}$ has no cycles! If not unique, it is often severely restrictive.

\begin{rem} \label{rmk:rescalingnoproblem} Note that rescaling the inclusion and projection maps (e.g. multiplying $p_j$ an invertible scalar $\kappa$, and multiplying $i_j$ by
$\kappa^{-1}$, for some $j$) will not change the graph, it will only rescale the edge labels. This is one symmetry we use freely below. \end{rem}

To illustrate this, let us return to the setting of categorified quantum groups. The quantum group has certain defining relations, such as $ef 1_{\lambda} = fe
1_{\lambda} + [\lambda] 1_{\lambda}$, which are usually categorified in $\dot{\Uc}$ by direct sum decompositions (for which Lauda wrote down some inclusion and projection maps
explicitly). Let us call these the \emph{defining direct sum decompositions} in $\dot{\Uc}$. If a $p$-differential $d$ on $\dot{\Uc}$ is to give rise to the correct Grothendieck group,
then at the least it should be the case that $\Gamma_{I,d}$ has no cycles for all the defining direct sum decompositions (for some choice of projection and inclusion maps). Let us temporarily call such a differential \emph{good}.

In \cite{EQpDGsmall,EQpDGbig} we first computed all the possible $p$-differentials on the category $\Uc$ and $\dot{\Uc}$. This produced a rather large family of differentials. Then
in \cite[Proposition 5.14]{EQpDGsmall} we computed which differentials in this family were good. That is, for each differential we computed the graphs $\Gamma_{I,d}$ for all the
defining direct sum decompositions (and all possible inclusion and projection maps), and asked which graphs had no cycles. Happily, being good is a strong enough condition to pin
down the differential and the inclusion and projection maps precisely! Ultimately, only two nonzero differentials $d$ and $\bar{d}$ were good, and these differentials are
intertwined by the duality functor (which flips diagrams upside-down). Thus already one deduces that there are only two dual $p$-differentials $d$ and $\bar{d}$ which could
possibly give rise to the correct Grothendieck group (caveat: see Remark \ref{rmk:caveat2}), though one still needs to confirm that they do have the correct Grothendieck group. 

In summary, that the defining direct sum decompositions lift to Fc-filtrations is typically a very restrictive property for a differential, and differentials which satisfy it are
quite special and interesting. Not only that, but the projection and inclusion maps compatible with these differentials should be considered as particularly nice.

\begin{rem} Perhaps most intriguingly, the special partial orders on the index sets $I$ induced by the cycle-less graph $\Gamma_{I,d}$ are new and unfamiliar structures which were
invisible before the introduction of the differential. To state a rough moral: categorification replaces structure coefficients (numbers) with multiplicity spaces (vector spaces),
while $p$-dg categorification equips these multiplicity spaces with a filtration, whose shadow is now combinatorial (an oriented graph). \end{rem}

\begin{rem} Essentially everything discussed in this section depended only on the differential as defined over $\Z$, but not on the choice of prime $p$ or on the fact that $d^p = 0$. This is not obvious, as specializing to characteristic $p$ could, in theory, eliminate cycles in the graph $\Gamma_{I,d}$, but in practice it does not. In other words, fantastic filtrations should really be considered as a theory intrinsic to gaeas. At the current moment, the homological algebra of gaeas has not been developed to the same degree as $p$-dg algebras were in \cite{QiHopf}, so whenever discussing the Grothendieck group we play it safe and talk only about $p$-dg algebras. \end{rem}

\subsection{Computing the Grothendieck group}

A useful tool towards computing the Grothendieck group has been the following result of Qi \cite{QiHopf}, an analogue of the positive dg algebra case by \cite{SchPos}.

\begin{prop} \label{prop:positive} In the special case when $A$ is a positively graded $p$-dg algebra\footnote{A positively-graded $p$-dg algebra has its grading supported in
non-negative degrees, with semisimple degree zero part, and the differential is trivial in degree zero.}, the $p$-dg Grothendieck group is just the
specialization at $q=\zeta$ of the original Grothendieck group. \end{prop}

When $A$ is not positively graded this result is often false, and great caution is required. Most interesting categorifications are not positively graded. In \cite{EQpDGsmall} we
developed Fc-filtrations as part of a game which manipulates a $p$-dg category until, hopefully, we can apply Proposition \ref{prop:positive}. Let us elaborate on this method in the context of $\dot{\Uc}$, which categorifies the quantum group $U_q(\mathfrak{sl}_2)$.

A common technique in the study of additive and abelian categories is to choose a projective generator and work instead with its endomorphism ring. Given a collection of self-dual
indecomposable objects we can study the full subcategory $\Pc$ in $\dot{\Uc}$ with those objects. Equivalently, letting $P$ be the direct sum of these self-dual indecomposable objects,
we can study the $p$-dg algebra $\End(P)$. Now the size of morphism spaces between objects in $\dot{\Uc}$ is determined by a particular sesquilinear pairing on $U_q(\mathfrak{sl}_2)$,
see \cite[\S 1.1]{LauSL2}. The self-dual indecomposable objects in $\dot{\Uc}$ categorify Lusztig's canonical basis, and the pairing of canonical basis elements has only non-negative
powers of $q$. Consequently, $\Pc$ is a positively-graded category, i.e., $\End(P)$ is a positively-graded algebra, from which one easily computes both the ordinary and the $p$-dg
Grothendieck group of $\Pc$. The task is to relate the $p$-dg category $\dot{\Uc}$ with the $p$-dg category $\Pc$ (the underlying additive categories are Morita equivalent, but the
$p$-dg setting is more subtle).

In the additive setting, $X$ is generated by $\Pc$ if it can be expressed as a direct sum of the objects in $\Pc$. To show that a given set of objects is a generator, we need to find
enough direct sum decompositions. In the $p$-dg world we need more; it is sufficient for $X$ to be filtered by objects in $\Pc$ via fantastic filtrations. That is, one other major
implication of a Fc-filtration as above is that $X$ will lie in the triangulated hull of $\{M_j\}$, inside the $p$-dg derived category.

\begin{rem} \label{rmk:caveat2} At the moment we do not have results on the necessity of Fc-filtrations. That is, in theory $X$ might be in the triangulated hull of $\{M_j\}$ even when
there is no fantastic filtration, or \eqref{desired} might hold, because no technology has been developed to provide an obstruction. However, we know of no examples where this happens.
 \end{rem}

There are other important direct sum decompositions, beyond the defining ones: for example, the idempotent decomposition known as the Stosic formula in \cite[Theorem 5.6]{KLMS}. The remainder of the argument in \cite{EQpDGbig} went as follows.
\begin{itemize} \item Find enough direct sum decompositions to decompose \emph{any} object in the category $\dot{\Uc}$ as a direct sum of objects which are either indecomposable or contractible. (In this case, the defining decompositions and the Stosic formula were sufficient.)
\item Prove that each of these direct sum decompositions lifts to a Fc-filtration, for the good differentials.
\item Let $P$ be the direct sum of all the non-contractible indecomposable objects. Deduce from the above that $P$ generates the $p$-DG derived category of $\dot{\Uc}$, and that $\End(P)$ is positive, so that the $p$-dg Grothendieck group of $\dot{\Uc}$ agrees with the $q = \zeta$ specialization of the ordinary Grothendieck group of $\Pc$. \end{itemize}

This method for computing the $p$-dg Grothendieck group has many obvious limitations. It worked for $\dot{\Uc}(\mathfrak{sl}_2)$ because the category is relatively simple. We
understand completely what all the indecomposable objects are, and we know enough explicit idempotent decompositions to take an arbitrary object and split it into indecomposables.
In his PhD thesis, Andrew Stephens \cite{StephensSL3} was able to extend this same method to categorify the positive half of quantum $\mathfrak{sl}_3$, because the explicit
idempotent decompositions were also constructed previously by Stosic \cite{StosicSL3}. It seems hopeless to extend this method to $\mathfrak{sl}_n$ in general: even the canonical
basis of the positive half of the quantum group is unknown, so there is little hope of understanding the indecomposable objects explicitly. Some new techniques are clearly required
to make progress beyond what is currently known.


\subsection{The diagrammatic Hecke algebra and its differential}

In this paper we initiate the program to categorify (Iwahori-)Hecke algebras at a root of unity, with preliminary positive results, useful negative results, and extremely
intriguing conjectures.

The Hecke algebra in type $A$ is a deformation of the group algebra of the symmetric group. It is categorified by the monoidal category of Soergel bimodules, which are certain
bimodules over the polynomial ring from Example \ref{ex:standarddiffpoly}. The appropriate integral form of this categorification is the diagrammatic Hecke category, as
introduced by Elias-Khovanov in type $A$ \cite{EKho}. Just as the symmetric group is generated by its simple reflections, the adjacent transpositions $s_i = (i,i+1)$, the
diagrammatic Hecke category is generated by certain objects $B_i = B_{s_i}$. Tensor products of these objects are commonly called \emph{Bott-Samelson} bimodules or objects. The
diagrammatic Hecke category encodes morphisms between Bott-Samelson bimodules as planar diagrams. A basis for these morphism spaces was constructed in \cite{EWGr4sb}, called the
\emph{double leaves basis}.

Following the motif from the parallel world of quantum groups, if we want to categorify the Hecke algebra at a prime root of unity, we should equip the diagrammatic Hecke category
with the structure of a gaea. In this paper, we only examine simply-laced type. As in the outline of \cite{EQpDGbig}, we first compute all possible
differentials on the category, obtaining a large family of even differentials. Then, for each of the defining idempotent decompositions, we compute the corresponding graph and
determine for which differentials the graph has no cycles. We temporarily call such a differential \emph{good}. Once again, this constraint is enough to pin down the differential
precisely: only two dual differentials ($d$ and $\bar{d}$) are good, and could possibly induce the correct $p$-dg Grothendieck group. We explicitly check all possible choices of
projection and inclusion maps for each of the defining idempotent decompositions.

\begin{rem} Some of the defining idempotent decompositions require dg-filtrations rather than Fc-filtrations, because the summands in question are not Bott-Samelson bimodules and
thus not pre-existing objects in the diagrammatic Hecke category. In fact, they require a slight generalization which mixes the concepts of a dg-filtration and a Fc-filtration,
see \S\ref{subsec-mixedFc}. \end{rem}

\begin{rem} A differential on the Hecke category induces a differential on its polynomial ring, which for the Elias-Khovanov version of the diagrammatic category is assumed to be
$\Bbbk[x_1, \ldots, x_n]$ with its standard action of $S_n$. This polynomial ring has a standard differential, where $d(x_i) = x_i^2$. In our classification we do not assume that
this is the differential on the polynomial ring. Instead, we prove that when the differential is good, the $p$-dg polynomial ring is forced to be isomorphic to the standard one.
Moreover, we prove that Hecke category, if constructed instead using the $(n-1)$-dimensional reflection representation of $S_n$, does not admit a good differential! \end{rem}

\begin{rem} There are a number of other compatibilities one might desire from a differential. The first is that it is compatible with the categorical Schur-Weyl duality of
Mackaay-Stosic-Vaz \cite{MSV}, which gives a functor from the Hecke category of $S_n$ to the categorification of quantum $\mathfrak{gl}_n$. The second is that the singular Hecke
2-category, which contains the Hecke category as an endomorphism category, should have a $p$-differential which restricts to our chosen differential. The third is that the thick
calculus of the first author \cite{EThick} should have a $p$-differential which restricts to our chosen differential. All these are satisfied by the good differential; we briefly discuss some of these compatibilities in this paper. \end{rem}

In practice, the classification of differentials is a three-step process. By the Leibniz rule, any differential on an algebra is determined by its action on the generators. A
generator is sent to some morphism of degree $2$ higher, living in a finite-dimensional morphism space, so we can specify this morphism by a number of parameters (its
coefficients in the double leaves basis). Now we impose three constraints on these parameters. The first is that the relations are preserved by the differential; this ensures
that the differential is well-defined. The second is that the divided power $d^{(k)}$ is well-defined over $\Z$;  this need only be checked on the generators, since
\begin{equation} d^{(k)}(fg) = \sum_{i+j = k} d^{(i)}(f) d^{(j)}(g) \end{equation}
by the Leibniz rule. The third
is that the defining idempotent decompositions should have no cycles. Only the first two constraints need to be checked when classifying general differentials rather than good
differentials.

As noted above, our main result is a classification of all good differentials on the Hecke category (up to object-fixing isomorphism): there are only two, $d$ and $\bar{d}$. We
also go most of the way towards classifying the general differential, although we do not quite finish the job. The general differential has many more non-zero parameters than the
good differential. For example, a general differential applied to the 4-valent vertex gives a sum of three diagrams with particular coefficients, while a good differential
applied to the 4-valent vertex is zero. Consequently, it is much easier to check that the good differential satisfies all the relations of the diagrammatic Hecke category than it
is to check the general case. We are able to check every relation in the general case except the most complicated one, the so-called Zamolodchikov relation associated to
parabolic subgroups of type $A_3$. Perhaps it is only laziness which prevents us from finishing this calculation, although it is a surprisingly thorny one. \footnote{ The interested reader is welcome to finish this calculation for us and write an appendix! That said, we are not sure why anyone should
care about the complete family of differentials anyway; the good differentials seem at the moment like the only interesting ones.}



The Hecke category was generalized to all Coxeter groups by Elias-Williamson in \cite{EWGr4sb}. The computations done here also have implications for possible differentials in
other types and for other realizations (see \cite[\S 3.1]{EWGr4sb}). In particular, we prove that there is no nonzero good differential for any realization in any simply-laced
type except for finite and affine type $A$ (our results are sufficient to classify the good differentials in affine type $A$ as well). Work in progress of the first author and
Lars Thorge Jensen is exploring differentials in some non-simply-laced types, using a realization which is central extension of the root realization.

Let us reiterate that the good differential has already been discovered in some sense, in the algebraic context of Soergel bimodules. Good differentials on the nilHecke algebra has
been studied in the work of Beliakova-Cooper \cite{BeCo} and Kitchloo \cite{Kitch}. In a different direction Khovanov and Rozansky in \cite{KRWitt} construct an action of the positive half of the Witt algebra $\WC^+$ on all
Bott-Samelson bimodules and on certain complexes thereof. This led to an action of $\WC^+$ on triply graded knot homology (just as
predecessors Beliakova-Cooper and Kitchloo defined an action of the Steenrod algebra on the characteristic $p$ versions).

The $\WC^+$ action equips Bott-Samelson bimodules with a differential. This differential on the objects (the Bott-Samelson bimodules) can be used in standard fashion to construct a
differential on morphism spaces between objects. A generator of the $\WC^+$ action gives rise to our good differential. Again, the existence of this differential is no surprise at
this point. Our paper has a different set of goals: to prove the key homological properties of $d$ (i.e. it is good), to prove the uniqueness of $d$ (the lack of other good
differentials), and to examine other types, en route to computing the $p$-dg Grothendieck group of the diagrammatic Hecke category.

\subsection{The Grothendieck group?}

Having classified the good differentials, the next step is to compute the $p$-dg Grothendieck group of the Elias-Khovanov category for these differentials. Now it is clear that the methods of \cite{EQpDGbig} will no longer suffice, for several reasons which we now discuss.

Indecomposable objects $B_w$ in the Hecke category (up to isomorphism and grading shift) are in bijection with elements $w$ in the symmetric group $S_n$, and they appear as direct
summands inside $B_{i_1} \ot \cdots \ot B_{i_d}$ whenever $s_{i_1} \cdots s_{i_d}$ is a reduced expression for $w$. Aside from these facts, the indecomposable objects $B_w$ are
extremely mysterious. Unlike the case of quantum $\mathfrak{sl}_2$, the size and structure of the indecomposable object $B_w$ depends on the characteristic of the base field! We write
${}^p B_w$ to indicate the dependence of this object on the characteristic. The smallest example where the size of ${}^p B_w$ depends on $p$ occurs when $p=2$ for 28 elements $w \in
S_8$, but examples become ever more frequent as $n$ grows larger.

In characteristic zero the indecomposable objects ${}^0 B_w$ categorify the Kazhdan-Lusztig basis, by results of Soergel \cite{Soer90}. The size of morphism spaces is determined by a
sesquilinear form on the Hecke algebra, and Kazhdan-Lusztig basis elements pair positively, so the endomorphism ring of $\bigoplus {}^0 B_w$ is positively graded. When $p$ is large
relative to $n$, ${}^p B_w$ will ``agree'' with ${}^0 B_w$ for all $w$, and continue to categorify the Kazhdan-Lusztig basis. Note that, as proven by Williamson \cite{WillCounter}, the
prime $p$ must grow at least exponentially with $n$ for this statement to hold! When $p$ is small relative to $n$, the the endomorphism ring of $\bigoplus {}^p B_w$ need not be
positively graded. This is the first nail in the coffin.

To decompose an arbitrary object in the Hecke category into indecomposables, we would need to be able to find the idempotent projecting to $B_w$ inside a reduced expression, and
would need to be able to decompose $B_w \ot B_s$ for each $w \in S_n$ and each simple reflection $s$. Essentially nothing is known about these idempotents, and it is an incredibly
difficult open problem to study them! Without an explicit idempotent decomposition, there are currently no tools which could prove that some decomposition lifts to an Fc-filtration
or a dg-filtration. This is the second nail in the coffin.

For all these reasons, we are forced to abandon the previous methods in $p$-dg theory and search for something new, which will be the focus of the sequel to this paper.

Aside from the method of \cite{EQpDGbig}, there are only a few other tools in the literature which can be used to compute the $p$-dg Grothendieck group. Recently, the second author
and Sussan \cite{QiSussan3} have developed the notion of a $p$-dg cellular algebra and a $p$-dg quasi-hereditary algebra, and proven that their $p$-dg Grothendieck groups are $q =
\zeta$ specializations of the ordinary Grothendieck group. However, their concepts only apply to cellular algebras over the base field $\Bbbk$, rather than cellular algebras over
other rings. In the lingo, their technology works for quasi-hereditary algebras but not for affine quasi-hereditary algebras in the sense of Kleschchev \cite{KleAHW}. The double
leaves basis is a cellular basis, but with base ring $R$ (the polynomial ring in $n$ variables), which itself has a nontrivial differential. One might hope that the techniques of
Qi-Sussan can be adapted to the more general setting, but this is not an easy adaptation. Still, the cellular structure on the Hecke category is one of the most powerful weapons in
the arsenal, and we expect that any successful approach will use it.

Despite the lack of tools, we still believe the end result.

\begin{conj} \label{conj:main} The diagrammatic Hecke category associated to $S_n$, when equipped with a good differential (either $d$ or $\bar{d}$) and specialized to characteristic $p$, has $p$-dg Grothendieck group isomorphic to the Hecke algebra at the appropriate root of unity. \end{conj}

\subsection{Computations for small $n$}

The methods of \cite{EQpDGbig} do suffice for small values of $n$, where we can compute all the idempotent
decompositions by hand, and where the size of the indecomposable objects does not depend on the characteristic. For each $w \in S_n$ and simple reflection $s$ with $ws > w$, the
direct summands of $B_w \ot B_s$ all\footnote{We are working in characteristic zero, or assuming that the object $B_w$ and its direct sum decomposition agrees with the
characteristic zero setting in the Grothendieck group. This holds when $n < 8$ in any characteristic.} have the form $B_z$ for various $z \in S_n$. We let $I_{w,s} \subset S_n$ denote the set
of such $z$ (with multiplicity). If the graph $\Gamma_{I_{w,s},d}$ has no cycles, then the decomposition \begin{equation} \label{BwBsBz} B_w B_s \cong \bigoplus_{z \in I_{w,s}} B_z \end{equation}
is fantastically filtered.

\begin{thm} \label{thm:Iwsnocycle} If $n \le 4$, $w \in S_n$ and $s \in S$, then $\Gamma_{I_{w,s},d}$ has no cycles. As a consequence, any Bott-Samelson splits into indecomposable objects via a fantastic
filtration, and Conjecture \ref{conj:main} holds for $n \le 4$. \end{thm}

The proof of this theorem is by direct, straightforward, and exhaustive computation, and we have chosen not to write it up. Preliminary calculations have verified that this theorem
continues to hold true for interesting examples with $n = 5$ as well.

Already one observes some mysterious phenomena. The graph $\Gamma_{I_{w,s},d}$, when it has no cycles, induces a special partial order $\prec$ on $I_{w,s}$. In order to prove that
$\Gamma_{I_{w,s},d}$ has no cycles in general, it would help to know in advance what this partial order will be.

For example, let $n=3$ and let $s = (12)$ and $t = (23)$ inside $S_3$. In the additive setting we know that
\[ B_{st} \ot B_s \cong B_{sts} \oplus B_s, \]
so that $d$ induces a partial order on the set $\{s,sts\}$. This partial order happens to be $s \prec sts$. Meanwhile, 
\[ B_{ts} \ot B_t \cong B_{sts} \oplus B_t, \]
with partial order\footnote{In both cases, $\bar{d}$ induces the opposite partial order.} $sts \prec t$. Already it is clear that $\prec$ is not determined by the Bruhat order. One might still pray that it is governed by some kind of lexicographic order or convex order on roots, but one would be disappointed.

Letting $u = (34)$ inside $S_4$, we can compute various other partial orders, including
\begin{subequations}
\begin{equation} us \prec usts \quad \text{from } I_{ust,s}, \end{equation}
\begin{equation} usts \prec stsuts \prec sutu \quad \text{from } I_{sutsu,t}, \end{equation}
\begin{equation} sutu \prec su \quad \text{from } I_{sut,u}. \end{equation}
\end{subequations}
There are no cycles in any of these
individual sets, but there is a cycle if the sets $I_{w,s}$ are all placed togther! In other words, there is no partial order on $S_n$ which restricts to the partial order on each
$I_{w,s}$. This makes the partial orders on $I_{w,s}$ especially mysterious.

\begin{rem} This cycle
\begin{equation} us \to usts \to stsuts \to sutu \to us\end{equation} is not the only one. One can replace $usts$ with $stsu$, or $sutu$ with $utus$, to obtain another cycle. In addition, there is a cycle
\begin{equation} tsu \to sts \to tstut \to tut \to tsu, \end{equation}
and one can replace $tstut$ with either $tutst$ or with $t$. \end{rem}

\subsection{Computations for $n=8$} \label{ssec:intron8}

In finite characteristic for larger values of $n$, as already noted, the decomposition of tensor products into indecomposable objects is different from in characteristic zero. Some
summands which split in characteristic zero will instead get ``glued together.'' This does not change the Grothendieck group itself! The Grothendieck group of the diagrammatic
Hecke category is the Hecke algebra in any characteristic. What it does change is the basis of the Grothendieck group given by the symbols of indecomposable objects. In
characteristic $p$, one defines ${}^p B_w$ as the top summand of the Bott-Samelson object associated to a reduced expression of $w$, i.e. the unique direct summand of this
Bott-Samelson which is not a direct summand of any shorter Bott-Samelson\footnote{One can prove that it is independent (up to isomorphism) of the choice of reduced expression.}.
These ${}^p B_w$ enumerate the isomorphism classes of indecomposable objects up to grading shift. The basis $[{}^0 B_w] = b_w$ is the usual Kazhdan-Lusztig basis, also called the
\emph{$0$-canonical basis}, while $[{}^p B_w] = {}^p b_w$ is now called the \emph{$p$-canonical basis}.

Let us give a concrete example, in the symmetric group $S_8$. Let $X$ denote the Bott-Samelson bimodule associated to the sequence of simple reflections
\[ (s_3, s_2, s_1, s_5, s_4, s_3, s_2, s_6, s_5, s_4, s_3, s_7, s_6, s_5), \]
which is a reduced expression for an element $w \in S_8$. Let $y = s_2 s_3 s_2 s_5 s_6 s_5$, which is the longest element of a parabolic subgroup. In characteristic zero (and any odd finite characteristic), there is a direct sum decomposition\footnote{Thanks to Lars Thorge Jensen for his computer calculations and his help finding this and other accessible examples. This decomposition was verified by his programs.}
\begin{equation} \label{baddecomp} X \cong B_w \oplus B_y. \end{equation}
However, in characteristic $2$, $X$ is indecomposable. What is happening is that the integral form of the Hecke category contains morphisms $p \co X \to B_y$ and $i \co B_y \to X$ such that $p \circ i = 2 \id_{B_y}$. These morphisms $p$ and $i$ span their respective Hom spaces over $\Z$. Now change base to a field. If $2$ is invertible then $\frac{ip}{2}$ is an idempotent projecting to $B_y$. If $2 = 0$ one can not construct any splitting of $i$ or $p$; consequently, $X = {}^2 B_w$ is indecomposable, and on the Grothendieck group, ${}^2 b_w = b_w + b_y$.

In similar fashion, it need not be the case that $\Gamma_{I_{w,s},d}$ has no cycles in order for Conjecture \ref{conj:main} to hold! Let us call an object \emph{$p$-dg
indecomposable} if it does not have any idempotents fitting into a dg-filtration. A refinement of Conjecture \ref{conj:main} might say that a Bott-Samelson associated to a reduced
expression (for $w$) has an Fc-filtration with a unique $p$-dg indecomposable summand ${}^d B_w$ not appearing inside any Fc-filtration of a shorter Bott-Samelson. In theory, these
objects ${}^d B_w$ might descend to a basis of the $p$-dg Grothedieck group, the \emph{$d$-canonical basis}. The $p$-dg Grothendieck group might still be the Hecke algebra, even if
the $d$-canonical basis is an unexpected one. For example, what if the graph for the decomposition \eqref{baddecomp} has cycles? Then $X$ has no idempotents fitting into a
dg-filtration, and $X = {}^d B_w$ is $p$-dg indecomposable. Morally speaking, this is no worse than having $X$ be indecomposable in characteristic $2$.

We did not bring up this specific example for no reason.

\begin{thm} The decomposition \eqref{baddecomp} is not Fc-filtered in any characteristic. If $i$ and $p$ are the morphisms discussed above, then $d(i) \ne 0$, $d(p) \ne 0$, and
$d(p)i \ne 0$ (so the graph has both a cycle and a loop). \end{thm}

The proof is by direct and nasty computation, and we do not write it up. In \S\ref{sec:n8example} we record the conclusion of our efforts for posterity, writing down $p$, $d(p)$,
$i$, $d(i)$, and $d(p)i$.

This theorem is very surprising, because it implies that (if Conjecture \ref{conj:main} is true) the $d$-canonical basis of the Hecke algebra is different from both the
$0$-canonical basis and the $p$-canonical basis! For $S_8$ in characteristic $p>2$, the $p$-canonical basis agrees with the $0$-canonical basis, but the $d$-canonical basis
contains ${}^d b_w = b_w + b_y$.

\begin{rem} Having just emphasized that it is not important that every direct sum decomposition is an Fc-filtration, we wish to re-emphasize that the defining direct sum decompositions (like $B_s B_s \cong B_s(-1) \oplus B_s(1)$) must be Fc-filtrations, or the ring structure on the $p$-dg Grothendieck group would be incorrect. \end{rem}

\subsection{A preview of part II}

As noted above, new techniques are required to compute the $p$-dg Grothendieck group of the Hecke category, and to understand Fc-filtrations in the absence of explicit
decompositions. In the paper \cite{EQsl2}, we introduce some new techniques (some of them conjectural) which we hope will fit the bill. Here is a quick preview.

First, we introduce what we call the \emph{counterdifferential}, a new structure which exists in both the Hecke and the quantum group settings. This is a derivation $z$ of degree
$-2$ (satisfying the Leibniz rule), giving a new gaea structure. Since most of the generating morphisms live in the minimal nonzero degree of their respective Hom spaces, $z$ will kill these generators. So $z$
is determined by what it does to the polynomial ring $R = \Bbbk[x_1, \ldots, x_n]$, where it sends $x_k \mapsto 1$ for each $k$. In the sequel we will prove the following basic structural results.
\begin{itemize} \item Letting $h$ denote the degree operator (which acts on morphisms of degree $k$ by multiplication by the scalar $k$), the triple $(d,h,z)$ acts as an $\mathfrak{sl}_2$-triple, making all Hom spaces into $\mathfrak{sl}_2$-representations.
	\item The triple $(\bard,h,z)$ is also an $\mathfrak{sl}_2$-triple.
	\item Each Hom space is a free $R$-module with its double leaves basis, as noted above. The $R$-span of any double leaf is preserved by $z$. Moreover, there is a partial order on the set of double leaves such that $d$ sends each double leaf to the $R$-span of double leaves which are weakly lower in the partial order. Thus the double leaves basis equips all Hom spaces with a very particular kind of filtration.
\end{itemize}

One should think of the $\mathfrak{sl}_2$-representations which appear in the Hom spaces as roughly being filtered by coVerma modules. After all, the polynomial ring in one
variable $\Bbbk[x]$ is precisely the coVerma module $\nabla(0)$ (the polynomial ring in multiple variables is more complicated). Hom spaces are large infinite-dimensional
modules, but may also have small finite-dimensional submodules. For example, the $\Bbbk$-span of the identity element $1$ inside $R$ is a finite-dimensional $\mathfrak{sl}_2$-subrepresentation. We propose that it is no accident that the (one-dimensional) finite part of $R$ agrees precisely with the span inside $R$ of the units!

To give another example, whenever $s \in S$ is a simple reflection, $\Hom(B_sB_sB_sB_s,B_s)$ is a free $R$-module of graded rank $q(q+q^{-1})^4$. Note that $B_s B_s B_s B_s$ splits
into $8$ shifted copies of $B_s$, with graded rank $(q+q^{-1})^3$, so one should expect to find $8$ projection maps inside $\Hom(B_s B_s B_s B_s, B_s)$. Amazingly, for any
fantastic filtration picking out these $8$ summands, the span of the $8$ projection maps will be an $8$-dimensional $\mathfrak{sl}_2$-subrepresentation of $\Hom(B_s B_s B_s B_s,
B_s)$, and this is the maximal finite-dimensional part of that infinite-dimensional representation.

More generally, we conjecture the following algorithm to find all the projection maps in a fantastic filtration on a Bott-Samelson bimodule $X$. Consider the $\mathfrak{sl}_2$-representation $V = \bigoplus_{w \in W} \Hom(X,B_w)$ of all morphisms to indecomposable objects (there are ways of modeling this representation without needing to understand the
indecomposable objects). Find the maximal finite-dimensional subrepresentation $F_1$; this should be spanned by projection maps. Then, take the quotient of $V$ by $J \cdot F_1$,
where $J$ is the Jacobson radical of the category. Call this quotient $V_1$. Now repeat, finding the maximal finite-dimensional subrepresentation of $V_1$, and so forth. Not only
do we conjecture that sufficiently many projection maps can be found inside these finite-dimensional subquotients, and that this will produce a fantastic filtration on $X$, but also that the existence of a fantastic filtration on $X$ is equivalent to the effectiveness of this procedure. Further details will await in the sequel to \cite{EQsl2}.

\begin{rem} One should think of this conjecture as a new kind of Hodge theory. The relative hard Lefschetz theorem (see \cite[Theorem 1.2]{EWrel} for the theorem in this context)
implies (in characteristic zero) that multiplicity spaces (appropriately defined) of an indecomposable summand in a Bott-Samelson bimodule satisfy the hard Lefschetz property with
respect to the appropriate Lefschetz operator, meaning that they are finite-dimensional $\mathfrak{sl}_2$-representations. Meanwhile, we are stating that the entire Hom space has the
structure of an $\mathfrak{sl}_2$-representation (with a very different raising operator), whose finite-dimensional $\mathfrak{sl}_2$-subrepresentation is related to the multiplicity
space. Note that $d$ is not a Lefschetz operator; it is more like an ``infinitesimal Lefschetz operator.'' Moreover, we also conjecture some positivity properties for $d$, analogous to
the relative Hodge-Riemann bilinear relations. \end{rem}

\begin{rem} The $\mathfrak{sl}_N$ link homologies are defined using categories of $\mathfrak{sl}_N$-foams. The diagrammatic Hecke category for $S_n$ admits a monoidal functor to
$\mathfrak{sl}_N$-foams for each $N$, see e.g. \cite{VazFoams, MackaayVazFoams}. The diagrammatic Hecke category for $S_n$ also has a monoidal ideal for each $2 \le N$. This ideal is zero when
$N > n$, and otherwise it is generated by the indecomposable object $B_{w_N}$, where $w_N$ is the longest element of $S_N$ after its typical embedding into $S_n$. The functor to
$\mathfrak{sl}_N$-foams should annihilate this monoidal ideal, though we are unsure where to find this statement, c.f. \cite[Chapter 22]{EMTW}. When $N=2$, the computations in this paper
imply that the monoidal ideal is preserved by the action of $d$ and $z$, so that the quotient still admits an action of $\mathfrak{sl}_2$ by derivations. We expect the same statement
to hold for any $N$.

Note furthermore that many link homologies (e.g. Khovanov homology) use a further monoidal quotient, where one also kills positive degree symmetric polynomials to obtain finite dimensional morphism spaces. This ideal is preserved by $d$, but not by $z$ in general. However, the ideal is preserved by $z$ in certain finite characteristics. In characteristic 2, the operator $z$ descends to the one used by Shumakovitch in \cite[Section 3]{Shu}, while in characteristic $p > 2$, it agrees with the $p$-nilpotent operator discovered by Wang \cite{JWang}. See also \cite{QRSW} for connections with the foam approach. \end{rem}

\paragraph{Acknowledgements.}

The first author was supported by NSF CAREER grant DMS-1553032, NSF FRG grant DMS-1800498, and by the Sloan Foundation. The second author was supported by the NSF grant
DMS-1947532 when working on this paper. Many thanks go to Lars Thorge Jensen for proofreading a previous version of this manuscript and finding numerous bugs, as well as for useful
conversations, and for his lovely computer programs, which helped us to find good examples to study in $S_8$. We wish to thank the referee for helpful comments.

\section{The answer} \label{sec-answer}

Let us state the end result of our computations. The rest of the paper will comprise the proof of these results.

We will not review the diagrammatic Hecke category here. See \cite{EWGr4sb} for details. We fix a Coxeter system $(W,S)$ with a realization, having polynomial ring $R$ with an action of
$W$. The degrees in $R$ are doubled, so that the simple roots $\{\alpha_s\}_{s \in S}$ have degree $2$. To avoid potential confusion, we refer to the simple roots and other homogeneous polynomials
of the same degree as \emph{linear polynomials} (rather than degree $2$ polynomials). Since $R$ is the endomorphism ring of the monoidal identity, a differential on the diagrammatic Hecke category induces a differential on $R$. One should not confuse the differential $d$ with the divided difference operators $\pa_s$ associated to each $s \in S$, see \eqref{eq:defdem}. In the pictures below, blue represents $s$, red represents some $t$ with $m_{st} = 3$, and green represents some $u$ with $m_{su} = 2$.

\begin{thm} Let $d$ be an even differential of degree $+2$ on the diagrammatic Hecke category in simply laced type. Then there exist linear polynomials $g_s, \barg_s \in R$ for each $s \in S$, such that the differential is defined on the generating diagrams by the following formulas.
\begin{subequations} \label{equationsfordiff}
\begin{equation}
d\left(~\ig{.75}{enddotblue}~\right) = {
\labellist
\small\hair 2pt
 \pinlabel {$g_s$} [ ] at 7 15
\endlabellist
\centering
\ig{.75}{space}
} \ig{.75}{enddotblue} 
\end{equation}
\begin{equation}
	d\left(~\ig{.75}{startdotblue} ~\right)= {
\labellist
\small\hair 2pt
 \pinlabel {$\barg_s$} [ ] at 7 15
\endlabellist
\centering
\ig{.75}{space}
} \ig{.75}{startdotblue} 
\end{equation}
\begin{equation}
	d\left(~\splitpoly{}{}{}~\right) = \splitpoly{}{}{-g_s}
	\end{equation}
	\begin{equation}
	d\left(~\mergepoly{}{}{}~\right) = \mergepoly{}{}{-\barg_s} \end{equation}
	\begin{equation}
	d\left(~\Xbgpoly{}{}{}{}~\right) = \Xbgpoly{-g_u}{}{}{} + \Xbgpoly{}{}{g_s}{} + \Xbgpoly{}{g_u - g_s}{}{}  
	\end{equation}
	\begin{equation}
    d\left(~\sixvalent~\right) = A \brokeXII + B \brokeVI + C \brokeII + D \brokeIV + \sixvalent \tallpoly{f}.
	\end{equation}
\end{subequations}
In this final formula, we have
\begin{subequations}
\begin{eqnarray} A &=& - \pa_s(g_t), \\
	B &=& -\pa_t(\barg_s), \\
	C &=& \pa_t(\barg_t) - \pa_t(g_s) - \pa_s(g_t), \\
	D &=& \pa_s(g_s) - \pa_s(\barg_t) - \pa_t(\barg_s), \\
	f &=& g_s - g_t - (\pa_s(g_s) + \pa_t(g_s)) \alpha_s + (\pa_t(g_t) + \pa_s(g_t)) \alpha_t. \end{eqnarray}
\end{subequations}
These formulas, together with a differential on $R$, determine the differential on the category. Let
\begin{equation} z_s = g_s + \barg_s \in R. \end{equation}
Then the differential on $R$ satisfies the following properties:
\begin{subequations} \label{propertiesofdiff}
\begin{equation} d(w(f)) = w(d(f)) \quad \text{ for all } w \in W, f \in R, \end{equation}
\begin{equation} d(\alpha_s) = \alpha_s z_s, \end{equation}
\begin{equation} z_s \in R^s \end{equation}
\begin{equation} z_s \in R^u \quad \text{ when } m_{su} = 2, \end{equation}
\begin{equation} s(\alpha_t) \pa_s(z_t) = \pa_s(\alpha_t) (z_s - z_t) \quad \text{ when } m_{st} > 2. \end{equation}
\end{subequations}

To state the converse, let us call the data of an even differential on $R$ and a collection of linear polynomials $g_s, \barg_s$ satisfying \eqref{propertiesofdiff} by the name of a \emph{potential differential}. Then a potential differential induces a differential on the diagrammatic Hecke category via the formulas \eqref{equationsfordiff} if and only if the $A_3$ Zamolodchikov relation is sent to zero. \end{thm}

In other words, we did not check the $A_3$ Zamolodchikov relation, and are unsure whether it is sent to zero by any potential differential, or if there are additional requirements to be met. We suspect there are no additional requirements.

\begin{defn} Let us call a differential $d$ on the Hecke category \emph{good} if the defining idempotent decompositions can be lifted to fantastic filtrations. \end{defn}

We will make this more precise later. In type $A$ the defining decompositions lift the relations
\begin{subequations}
\begin{equation} b_s b_s = v b_s + v^{-1} b_s, \end{equation}
\begin{equation} b_s b_u = b_u b_s \qquad \text{if } m_{su} = 2, \end{equation}
\begin{equation} b_s b_t b_s - b_s = b_t b_s b_t - b_t \qquad \text{if } m_{st} = 3 \end{equation}
\end{subequations}
in the Hecke algebra. Note that the zero differential is good.

\begin{thm} A differential is good (in simply laced type) if and only if it satisfies the following properties. \begin{enumerate}
\item These equations hold.
\begin{subequations}
\begin{equation} d(g_s) = g_s^2, \end{equation}
\begin{equation} \barg_s = s(g_s), \end{equation}
\begin{equation} u(g_s) = g_s \quad \text{whenever } m_{su} = 2, \end{equation}
\begin{equation} st(g_s) = g_t \quad \text{whenever } m_{st} = 3. \end{equation}
\end{subequations}
This last equation implies that if $g_s = 0$ for some $s$, then $g_t = \barg_t = 0$ for all $t$ in the same connected component of the Coxeter graph.
\item If $g_s \ne 0$ then $g_s \notin R^s$.
\item If $g_s \ne 0$ and $m_{st} = 3$ then exactly one of these two possibilities holds. We encode which one holds using an orientation on the corresponding edge in the Coxeter graph.
\begin{enumerate}
\item $g_t$ is fixed by $s$, and $g_s = t(g_t)$ is fixed by $sts$. We orient the edge from $t$ to $s$.
\item $g_s$ is fixed by $t$, and $g_t = s(g_s)$ is fixed by $sts$. We orient the edge from $s$ to $t$.
\end{enumerate}
\item The orientation is \emph{consistent} in that, for any parabolic subgroup of type $A_3$, the middle vertex is neither a source nor a sink.
\end{enumerate}
Finally, a potential differential which satisfies the properties listed in this theorem will send the $A_3$ Zamolodchikov relation to zero, so it does induce a differential on the diagrammatic Hecke category.
\end{thm}

\begin{cor} The only connected simply-laced Coxeter groups which admit a consistent orientation have type $A_n$, type $\tilde{A}_n$, or type $A_{\infty}$, and they each have
precisely two consistent orientations. \end{cor}

\begin{proof} There is no way to consistently orient the $D_4$ Coxeter graph, which is contained inside any connected simply-laced Coxeter group aside from those listed above. \end{proof}

When the differential is good, the formula for the differential simplifies. The scalar $\kappa = \pa_s(g_s)$ is independent of the choice of $s$ in a connected component of the Coxeter graph.  We have
\begin{subequations}
\begin{equation} d\left(~\Xbgpoly{}{}{}{}~\right) = 0, \end{equation}
\begin{equation} d\left(~\sixvalent~\right) = 0 \quad \text{if } t \to s, \end{equation}
\begin{equation} d\left(~\sixvalent~\right) = \kappa \left(~\brokeXII - \brokeVI ~\right) \quad \text{if } s \to t, \end{equation}
and can deduce that
\begin{equation} d\left(~\cupblue~\right) = -\kappa \startdotblue \startdotblue, \qquad d\left(~\capblue ~\right) = \kappa \finaldotblue \finaldotblue. \end{equation}
\end{subequations}

Finally, we prove that there are only two good differentials up to equivalence.

\begin{thm} Up to an automorphism of the Hecke category in type $A_n$, if there is a good derivation then we can assume that $R$ contains the polynomial ring  $\Bbbk[x_1, \ldots, x_n]$ with its usual $S_n$ action and differential $d(x_i) = x_i^2$, and we can assume that either $g_{(i,i+1)} = x_i$ for all $i$ with $\kappa = 1$, or that $g_{(i,i+1)} = x_{i+1}$ for all $i$ with $\kappa = -1$. An exception is when $\Bbbk$ has characteristic $2$, in which case it is possible that one may need to impose the relation $\sum x_i = 0$. \end{thm}

\section{Preliminaries} \label{sec-prelims}

We start by fixing a realization of a Coxeter system $(W,S)$. For more on realizations, see \cite[Section 3.1]{EWGr4sb}. Later we will assume that the Coxeter system is simply-laced, but not at first.

Recall that $\pa_s \co R \to R$ denotes the Demazure operator of a simple reflection $s \in S$, which is defined by the formula
\begin{equation} \label{eq:defdem} \pa_s(f) \alpha_s = f - s(f), \end{equation} and satisfies the twisted Leibniz rule
\begin{equation} \pa_s(fg) = \pa_s(f) g + s(f) \pa_s(g). \end{equation}
The Cartan matrix encodes the values of $\pa_s(\alpha_t)$ for various simple reflections.

We always assume that Demazure surjectivity holds, see \cite[Assumption 3.7]{EWGr4sb}. The implication is that for each $s \in S$ there exists some linear polynomial $\om_s$ such that
\begin{equation} \pa_s(\om_s) = 1. \end{equation} For example, if $m_{st} = 3$ we could set $\om_s = - \alpha_t$. We do not assume that $\om_s$ is a fundamental weight (i.e. that $\pa_t(\om_s) = 0$ for $t \ne s$).

 We also assume (c.f. \cite[Definition 1.7]{EWSFrob}) that \begin{equation} \nonumber (\star): \text{ the linear terms in $R^s$ and the linear terms in
$R^t$ together span the linear terms in $R$, whenever $s \ne t$}. \end{equation} This rules out the degenerate possibility that the fixed hyperplanes of $s$ and $t$ coincide. Both
Demazure surjectivity and $(\star)$ are requirements for the Hecke category to behave optimally (though they are sometimes replaced by related assumptions, like reflection faithfulness).

Let us try to define a differential on the Hecke category, by defining it as generally as possible, and then determining what constraints are imposed by the fact that it must preserve
the relations of the category. We will simultaneously determine what additional constraints are imposed if we want the differential to be good. We do not assume that our differential
is invariant under the symmetries of the Hecke category (vertical and horizontal flips, rotation by 180 degrees, Dynkin diagram automorphisms)\footnote{It will turn out that any differential must be invariant under horizontal flip, but not typically under the other symmetries.}.

The reader interested in type $A$ can fix $n \ge 2$ and work with the standard realization associated to $S_n$. The base ring $R$ has the form
\begin{equation}
 R = \Bbbk[x_1, \ldots, x_n] 
\end{equation}
with its usual action of $S_n$, and $\deg x_i = 2$. There is a \emph{standard differential} on the polynomial ring $R$, namely
\begin{equation} 
d(x_i) = x_i^2 
\end{equation}
for all $1 \le i \le n$. We will not assume that this is the induced differential on $R$, although this will eventually be a consequence of our computations.

Below, $s$ will be an arbitrary simple reflection and will be drawn using the color blue, $u$ will be distant from $s$ and drawn as green, and $t$ will be either be another arbitrary
simple reflection, or will be adjacent to $s$ ($m_{st} = 3$) and drawn using the color red.

\section{One color considerations} \label{sec-onecolor}

\subsection{Dots and polynomials} \label{subsec-dotspolys}

One of the generating morphisms in the diagrammatic Hecke category is the $s$-colored enddot $\finaldotblue$, a morphism $B_s \to \1$ of degree $+1$.
The enddot generates $\Hom(B_s,\1)$ as a free rank $1$ module over $R$. Thus for any differential we have
\begin{equation} \label{eq:denddot} d(\finaldotblue) = \poly{g_s} \finaldotblue \end{equation}
for some linear polynomial $g_s \in R$.

Consider the $s$-colored startdot $\startdotblue$, a morphism $\1 \to B_s$ of degree $+1$. For similar reasons we have
\begin{equation}\label{eq:dstartdot} d(\startdotblue) = \poly{\barg_s} \startdotblue \end{equation}
for some linear polynomial $\barg_s \in R$.

\begin{rem} We will eventually see that the differential on the entire category is determined by the values of $g_s$ and $g_s'$. \end{rem}

Now we compute that
\begin{equation} d(\barbblue) = \longpoly{(g_s + \barg_s)} \barbblue. \end{equation}
Since the barbell is equal to multiplication by $\alpha_s$
\[ \barbblue = \poly{\alpha_s}, \]
we deduce that
\begin{equation} d(\alpha_s) = \alpha_s(g_s + \barg_s). \end{equation}
It turns out that the linear polynomial $g_s + \barg_s$ is more intrinsic than either $g_s$ or $\barg_s$. Henceforth we write
\begin{equation} z_s = g_s + \barg_s,\end{equation}
so that
\begin{equation} \label{eq:dalpha} d(\alpha_s) = \alpha_s z_s \end{equation}
and
\begin{equation} d(\startdotblue) = \longpoly{(z_s - g_s)} \startdotblue. \end{equation}

If the differential on $R$ is known, then $z_s$ is determined by \eqref{eq:dalpha}. Conversely, the differential on $R$ must be such that $d(\alpha_s)$ is a multiple of $\alpha_s$,
which is not true of the most general differential.

\begin{rem} When $s = (i,i+1)$ we have $\alpha_s = x_i - x_{i+1}$. The standard differential satisfies
\begin{equation} d(x_i - x_{i+1}) = x_i^2 - x_{i+1}^2 = (x_i - x_{i+1})(x_i + x_{i+1}) \end{equation}
and thus $z_s = x_i + x_{i+1}$. Later on we will deduce that for a good differential one has $g_s = x_i$ and $\barg_s = x_{i+1}$, or vice versa (this will be the difference between $d$ and $\bar{d}$). Again, we will not assume anything about the differential on $R$ or $z_s$ or $g_s$ just yet, but we recommend that the reader verify the formulas below with these specializations, as a motivational sanity check. \end{rem}

Now we examine the polynomial forcing relations. We have
\begin{equation} \label{eq:polyforce} \poly{f} \lineblue = \brokenblue \longpoly{\pa_s(f)} + \lineblue \longpoly{s(f)}. \end{equation}
Applying the differential to both sides we have
\begin{equation} \longpoly{d(f)} \lineblue = \brokenblue \longpoly{z_s \pa_s(f)} + \brokenblue \longpoly{d(\pa_s(f))} + \lineblue \longpoly{d(s(f))}. \end{equation}
We resolve the LHS by applying \eqref{eq:polyforce} again. In the result, the equality of coefficients of $\brokenblue$ is
\begin{equation} \label{eq:dpaI} \pa_s(d(f)) = z_s \pa_s(f) + d(\pa_s(f)), \end{equation}
and the equality of coefficients of $\lineblue$ is
\begin{equation} \label{eq:sdcommute} s(d(f)) = d(s(f)). \end{equation}
By \eqref{eq:sdcommute} we deduce that the differential $d$ commutes with the action of the symmetric group $S_n$. A consequence is that $d$ preserves the invariant subring $R^s$, for all simple reflections $s$.

In fact, \eqref{eq:dpaI} already follows from the fact that $d$ commutes with the symmetric group action. By definition of the Demazure operator we know that
\[ \alpha_s \pa_s(f) = f - s(f), \]
which appeared above as \eqref{eq:defdem}. Taking the differential of both sides of \eqref{eq:defdem} and using \eqref{eq:sdcommute} we get that
\[ \alpha_s z_s \pa_s(f) + \alpha_s d(\pa_s(f)) = d(f) - s(d(f)) = \alpha_s \pa_s(d(f)). \]
Dividing both sides by $\alpha_s$ (a non-zero-divisor in $R$) we recover \eqref{eq:dpaI}.

In the special case when $f = \alpha_s$, \eqref{eq:dpaI} reduces\footnote{We will often use without mention the fact that $d$ kills any scalar multiple of the identity, and hence any polynomial of degree zero. This is a consequence of the Leibniz rule for differentials.} to
\[ \pa_s(\alpha_s z_s) = 2 z_s.\]
However, we already knew from the twisted Leibniz rule that
\[ \pa_s(\alpha_s z_s) = 2 z_s - \alpha_s \pa_s(z_s)\]
from which we deduce that
\begin{equation} \label{eq:pazsame} \pa_s(z_s) = 0, \qquad z_s \in R^s. \end{equation}

We could have deduced that $z_s \in R^s$ in a different way. Since $\alpha_s^2 \in R^s$, and $R^s$ is preserved by $d$, we see that
\[ d(\alpha_s^2) = 2 \alpha_s^2 z_s \in R^s, \]
which implies that $z_s \in R^s$. We have brought up this alternative method because we can use it for other purposes as well. For example, suppose that $m_{su} = 2$, so that $s$ and $u$ are \emph{distant} simple reflections. Since $\alpha_s \in R^u$, we must have $d(\alpha_s) \in R^u$, from which we deduce that
\begin{equation} z_s \in R^u \quad \text{ when } m_{su} = 2. \end{equation}

Now consider the case when $f = \alpha_t$ for some $s \ne t$, so that $\pa_s(f)$ is some scalar. Then \eqref{eq:dpaI} reduces to
\begin{equation} \pa_s(\alpha_t z_t) = \pa_s(\alpha_t) z_s. \end{equation}
Applying the twisted Leibniz rule, we get
\begin{equation} \label{eq:proportional} s(\alpha_t) \pa_s(z_t) = \pa_s(\alpha_t) (z_s - z_t). \end{equation}
Both sides are zero when $m_{st} = 2$. When $m_{st} > 2$, so that $\pa_s(\alpha_t)$ is nonzero, we deduce that $z_s - z_t$ is proportional to the root $s(\alpha_t)$.

In the special case when $m_{st} = 3$ we find that
\begin{equation} \label{eq:prop3}(\alpha_s + \alpha_t) \pa_s(z_t) = z_t - z_s. \end{equation}
Applying $\pa_t$ to both sides of the equation, we get
\begin{equation} \label{eq:zszt} \pa_s(z_t) = - \pa_t(z_s). \end{equation}

\begin{rem} The first major consequence of \eqref{eq:proportional} is that one need not have an incredibly large realization in order to find a $p$-differential $d$. If the Coxeter
graph is connected, then the subspace spanned by $\{\alpha_s,z_s\}_{s \in S}$ is at most one dimension higher than the subspace spanned by the roots. One should think that there is a
realization spanned by the simple roots and a single new element $z$, and that each $z_s$ can be written as a linear combination of the roots and $z$. Let $R'$ denote the subring of $R$
generated by the roots and $z$. Then, so long as $d(z) \in R'$, our computations will imply that the Hecke category, when defined over $R'$, will be preserved by the differential. In the standard setup for $S_n$, we might let $z = x_1$, for example.
\end{rem}

\begin{rem} We have imposed many conditions on the elements $z_s$ and on the differential, making the situation seem quite overdetermined, and the reader may already be convinced that in type $A$ the standard differential on $R = \Bbbk[x_1, \ldots, x_n]$ is the only one which could satisfy them all. This intuition is entirely false. An arbitrary $S_n$-invariant differential on $R$ has the form
\begin{equation} d(x_i) = A x_i^2 + B x_i \sum_{j \ne i} x_j + C \sum_{j \ne i} x_j^2 + D \sum_{j < k, j \ne i \ne k} x_j x_k, \end{equation}
for some scalars $A, B, C, D$. For any such differential, one can compute that $d(x_i - x_j) = z_{ij}(x_i - x_j)$ where
\begin{equation} z_{ij} = (A-C)(x_i + x_j) + (B - D)\sum_{k \ne i, j} x_k. \end{equation}
One can verify \eqref{eq:proportional} in this general setting. So, in fact, the constraints on $z_s$ which we have deduced above all are consequences of the fact that $d$ is $S_n$-invariant.

The extra condition that $d^p=0$ in characteristic $p$ will restrict the situation considerably, but it is not clear precisely by how much. \end{rem}

\subsection{Trivalent vertices} \label{subsec-trivalent}

We now compute the differential applied to the trivalent vertices $\splitblue$ and $\mergeblue$, which have degree $-1$.

It is not difficult to prove that every degree $+1$ map $B_s \to B_s \ot B_s$ is a linear combination of taking $\splitblue$ and adding a linear polynomial to one of the three regions. (This is because any element of the double leaves basis can be obtained from $\splitblue$ by breaking some lines.) Consequently, let us assume that
\begin{equation} \label{eq:dsplitassump} 
d \left(~\splitpoly{}{}{}~\right) = \splitpoly{f_1}{}{} + \splitpoly{}{f_2}{} + \splitpoly{}{}{f_3}. 
\end{equation}
	
\begin{rem} \label{rem:sliding} This description of a degree $+1$ morphism is not unique, in that different values of $f_1, f_2, f_3$ can give rise to the same morphism. If we wanted to make it unique, we could assert that $f_1$ and $f_2$ are each scalar multiples of $\om_s$. To prove this, set
\begin{equation} g_1 = f_1 - \pa_s(f_1) \om_s. \end{equation}
Then $\pa_s(g_1) = 0$ so $g_1 \in R^s$. So, sliding $g_1$ across the $s$-wall, we can replace $f_1$ with $\pa_s(f_1) \om_s$, and replace $f_3$ with $f_3 + g_1$. We can do the same with $f_2$. \end{rem}

Consider the counit relation
\begin{equation} \label{eq:counitright} \ig{1}{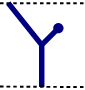} = \linebluelabel{}. \end{equation}
Applying the differential to both sides, we get
\begin{equation} 0 = \poly{f_1} \lineblue + \lineblue \poly{f_2} + \lineblue \poly{f_3} + \lineblue \poly{g_s}. \end{equation}
If $f_1 \notin R^s$ then this equation can not possibly hold. If $f_1 \in R^s$ we could rewrite \eqref{eq:dsplitassump} such that $f_1 = 0$ by setting the new $f_3$ equal to the old $f_1 + f_3$. So we can assume that $f_1 = 0$. But the horizontal reflection (i.e. flip across a vertical axis) of \eqref{eq:counitright} also holds, from which we deduce that (without loss of generality) $f_2 = 0$ as well. Finally, we see that $f_3 = - g_s$. Thus
\begin{equation} \label{eq:dsplit} 
d\left(~\splitpoly{}{}{}~\right) = \splitpoly{}{}{-g_s}.
 \end{equation}

An almost identical computation with the unit relation will give that
\begin{equation} \label{eq:dmerge} 
d \left(~\mergepoly{}{}{}~\right) = \mergepoly{}{}{-\barg_s}. 
\end{equation}
	
We have now pinned down the differential on the trivalent vertices, given the choice of $g_s$ and $\barg_s$. Note that the differential of the cup and cap are not trivial. Using the fact that a cup or cap is the composition of a dot and a trivalent vertex, and forcing the polynomials outside of the cup or cap, we get that
\begin{equation} \label{eq:dcap}
 d \left(~\capblue~\right) = \pa_s(g_s) \finaldotblue \finaldotblue + \capblue \longpoly{g_s - s(\barg_s)}, \end{equation}
\begin{equation} \label{eq:dcup} 
d \left(~\cupblue~\right) = -\pa_s(g_s) \startdotblue \startdotblue + \cupblue \longpoly{\barg_s - s(g_s)}.\end{equation} 
In \eqref{eq:dcap} we have used that $\pa_s(-\barg_s) = \pa_s(g_s)$.
	
We should now check that the remaining one-color relations are satisfied for any such differential. Consider the needle relation
\begin{equation} \label{eq:needle} \needlepoly{} = 0. \end{equation} Applying the differential to the LHS we get
\begin{equation} 
d \left(~\needlepoly{}~\right) = \needlepoly{-z_s} = (-z_s) \cdot \needlepoly{} = 0 \end{equation}
since $z_s \in R^s$. This agrees with the differential applied to the RHS, as desired.

The remaining one-color relations (namely: associativity, coassociativity, and Frobenius associativity) are all easily checked.

\subsection{The quadratic relation and its defining idempotent decomposition} \label{subsec-bsbsdecomp}

The quadratic relation $b_s^2 = [2]b_s$ in the Hecke algebra is categorified by a direct sum decomposition
\begin{equation} \label{BsBsdecomp} B_s B_s \cong B_s(1) \oplus B_s(-1) \end{equation}
in the Hecke category. Here $(1)$ represents the grading shift. To specify the convention: within the shifted bimodule $R(1)$, the identity element lives in degree $-1$.

To show this direct sum decomposition \eqref{BsBsdecomp} we should choose projections $p_{+1},p_{-1}$ and inclusions $i_{+1},i_{-1}$ from the two factors $B_s(+1)$ and $B_s(-1)$ respectively. Let $I = \{+1,-1\}$ be the indexing set for this decomposition.

The projection map $p_{-1}$ must be a scalar multiple of $\mergeblue$, for degree reasons. Similarly, the inclusion $i_{+1}$ must be a scalar multiple of $\splitblue$. These scalars must be invertible in order for \eqref{directsumdecomp} to hold, and rescaling the projection and inclusion maps will not change the graph $\Gamma_{I,d}$, see Remark \ref{rmk:rescalingnoproblem}. Thus we can assume that $p_{-1} = \mergeblue$ and $i_{+1} = \splitblue$ precisely.

The projection map $p_{+1}$ must have the form
\begin{equation} p_{+1} = A \finaldotblue \lineblue + B \ig{.5}{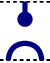} + \mergeblue \poly{f} \end{equation}
for some linear polynomial $f$ and some scalars $A$ and $B$, since this is a general morphism of degree $+1$ written in the double leaves basis. Similarly, we have
\begin{equation} i_{-1} = A' \startdotblue \lineblue + B' \ig{.5}{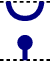} + \splitblue \poly{f'}. \end{equation}

In order for \eqref{directsumdecomp} to hold, we need
\begin{subequations}
\begin{equation} A' = 1 = A, \end{equation}
\begin{equation} \label{BB2} B' + B = -2, \end{equation}
\begin{equation} \label{foo5} f' + f = \alpha_s. \end{equation}
\end{subequations}
For this computation we have used the formula
\begin{equation} \lineblue \lineblue = \startdotblue \mergeblue + \finaldotblue \splitblue - 2 \ig{.5}{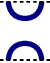} + \ig{.5}{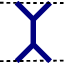} \poly{\alpha_s}. \end{equation}

Recall that the graph $\Gamma_{I,d}$ has edges labeled by $p_j d(i_k)$ for each $j,k \in I$, but with zero-labeled edges removed. We now determine, for any possible choice of projections and inclusions, when the graph the graph $\Gamma_{I,d}$ has no loops and cycles. Let us set $\kappa_s = \pa_s(g_s)$, which will appear because we use
\begin{equation} \needlepoly{g_s} = \kappa_s \lineblue \end{equation}
frequently in this computation. In this section we shorten $\kappa_s$ to $\kappa$. Note also that $\pa_s(\barg_s) = - \kappa$, since $g_s + \barg_s = z_s$ is killed by $\pa_s$.

We have
\begin{subequations}
\begin{equation} p_{+1} d(i_{+1}) = \longpoly{-g_s} \lineblue - B \kappa \brokenblue - \lineblue \kappa f = - \brokenblue \longpoly{(1+B)\kappa} - \lineblue \longpoly{s(g_s) + \kappa f} \end{equation}
\begin{equation} p_{-1} d(i_{-1}) = \poly{\barg_s} \lineblue - B' \kappa \brokenblue - \lineblue \kappa f' = \brokenblue (1+B')(-\kappa) + \lineblue \longpoly{s(\barg_s) - \kappa f'}. \end{equation}
\end{subequations}
In particular, the graph $\Gamma_{I,d}$ will have a loop at vertex $+1$ unless
\begin{subequations} \label{foo1-4} 
\begin{equation} \label{foo1} (1+B) \kappa = 0, \end{equation}
\begin{equation} \label{foo2} s(g_s) + \kappa f = 0, \end{equation}
and a loop at vertex $-1$ unless
\begin{equation} \label{foo3} (1+B') \kappa = 0, \end{equation}
\begin{equation} \label{foo4} s(\barg_s) - \kappa f' = 0. \end{equation}	
\end{subequations}

Note that \eqref{foo3} and \eqref{foo1} are equivalent via \eqref{BB2}. In order for \eqref{foo1-4} to hold, we have two options.

Suppose $\kappa = 0$ so that $g_s \in R^s$. Then \eqref{foo2} and \eqref{foo4} hold if and only if $g_s = \barg_s = 0$. Thus also $z_s = 0$. Then the differential is zero on every $s$-colored diagram: dots, trivalent vertices, and barbells (i.e. $\alpha_s$).

Now suppose $\kappa \ne 0$, so that $g_s \ne s(g_s) \ne 0$. Then $B = B' = -1$ in order for \eqref{foo1} to hold. Moreover, combining \eqref{foo2} and \eqref{foo4} and \eqref{foo5} we get
\begin{equation} \alpha_s = f + f' = \frac{s(\barg_s - g_s)}{\pa_s(g_s)} \end{equation}
or equivalently (since $s(\alpha_s) = -\alpha_s$)
\begin{equation} \pa_s(g_s) \alpha_s = g_s - \barg_s. \end{equation}
However, by definition of the Demazure operator we have
\begin{equation} \pa_s(g_s) \alpha_s = g_s - s(g_s) \end{equation}
and we conclude that
\begin{equation} \barg_s = s(g_s). \end{equation}
In addition to this interesting constraint on the differential, the idempotent decomposition (for which the graph has no cycles) is uniquely determined, since
\begin{equation} f = \frac{-\barg_s}{\kappa}, \quad \quad f' = \frac{g_s}{\kappa}. \end{equation}
In fact, we encourage the reader to doublecheck the following intriguing equality:
\begin{equation}\label{intriguing} \kappa p_{+1} = d(p_{-1}), \qquad -\kappa i_{-1} = d(i_{+1}). \end{equation}
We discuss why it is intriguing in the next section.

We have just ensured that $\Gamma_{I,d}$ has no loops, but we also need to check that it has no cycles, which means checking that either $p_{-1} d(i_{+1}) = 0$ or $p_{+1} d(i_{-1}) = 0$.

We have \begin{equation} p_{-1} d(i_{+1}) = - \kappa \lineblue. \end{equation} Consequently, this edge is zero in the $\kappa = 0$ case, and nonzero otherwise. If $\kappa \ne 0$, then the remaining edge $p_{+1} d(i_{-1})$ had better be zero, or there will be a cycle.

Now for the nastiest computation. We only give the answer; the $*$s represent two different polynomials.
\begin{subequations}
\begin{equation} p_{+1} d(i_{-1}) = \brokenblue \poly{*} + \lineblue \poly{*}. \end{equation}
The coefficient of $\brokenblue$ is 
\begin{equation} \label{foo10} -\kappa(B' f + B f' + B B' \alpha_s) \end{equation}
and the coefficient of $\lineblue$ is
\begin{equation} d(f') - \kappa f f' - f' z_s. \end{equation}
\end{subequations}

When $\kappa \ne 0$, the requirement that $B = B' = -1$ from \eqref{foo1}, combined with \eqref{foo5}, will already imply that \eqref{foo10} is zero. Thus when $\kappa \ne 0$ and there are no loops in $\Gamma_{I,d}$, then there are no cycles in $\Gamma_{I,d}$ if and only if
\[d(f') = \kappa f f' + f' z_s. \]
Plugging in the known values of $f$ and $f'$ we get
\[ \frac{d(g_s)}{\kappa} = \frac{g_s(z_s - \barg_s)}{\kappa} = \frac{g_s^2}{\kappa} \]
or in other words
\begin{equation} \label{eq:dgs} d(g_s) = g_s^2. \end{equation}
Here we have found a very interesting constraint! Of course, the other case when $g_s = 0$ also satisfies this constraint.

In conclusion, we have proven the following result.

\begin{prop} \label{prop:onecolorgood} There is a fantastically filtered idempotent decomposition $B_s B_s \cong B_s(+1) \oplus B_s(-1)$ if and only if $d(g_s) = g_s^2$, $\barg_s = s(g_s)$, and one of the following possibilities holds: \begin{itemize}
\item $g_s = 0$, or
\item $g_s \notin R^s$.
\end{itemize} \end{prop}

In both cases, the differential of the cap and cup has a simplified formula.
\begin{equation} 
d\left(~\cupblue ~\right) = -\kappa_s \startdotblue \startdotblue, \qquad d
\left(~\capblue~\right) = \kappa_s \finaldotblue \finaldotblue, \end{equation}
where $\kappa_s = \pa_s(g_s)$.
	
\subsection{An alternate approach} \label{subsec-alternate}

The reader should confirm that the condition that $d(g_s) = g_s^2$ is equivalent to the condition that $d(d(i_{+1})) = 0$. We want to briefly elaborate on the relationship between
$d^2(i_{+1}) = 0$ and \eqref{intriguing}; namely, we will prove the existence of a Fc-filtration using only two spare assumptions, with no further computation!

So let us begin again in our attempt to construct a Fc-filtration, noting once more that $i_{+1}$ and $p_{-1}$ are forced to be trivalent vertices (after rescaling) because they live in one-dimensional Hom spaces. Our two assumptions are:
\begin{subequations}
\begin{equation} \label{ddizero} d(d(i_{+1})) = 0, \end{equation}
\begin{equation} \label{pdinonzero} p_{-1} d(i_{+1}) = \lambda \id_{B_s} \end{equation}
\end{subequations}
where $\lambda$ is some invertible scalar. We claim that setting $i_{-1} := \lambda^{-1} d(i_{+1})$ and $p_{+1} := -\lambda^{-1} d(p_{-1})$ will yield a Fc-filtration.

First we check that this gives an idempotent decomposition, and our computation is just an exercise in the Leibniz rule. It is clear that $p_{-1} i_{+1} = 0$ for degree reasons, and $p_{-1} i_{-1} = \id_{B_s}$ is given by \eqref{pdinonzero}. Now
\begin{equation} p_{+1} i_{+1} = -\lambda^{-1} d(p_{-1}) i_{+1} = -\lambda^{-1} (d(p_{-1} i_{+1}) - p_{-1} d(i_{+1})) = -\lambda^{-1}(-\lambda \id_{B_s}) = \id_{B_s}. \end{equation}
Here we used the Leibniz rule together with the fact that $p_{-1} i_{+1} = 0$. Continuing,
\begin{equation} p_{+1} i_{-1} = -\lambda^{-1} d(p_{-1}) i_{-1} = -\lambda^{-1}(d(p_{-1} i_{-1}) - p_{-1} d(i_{-1})) = 0. \end{equation}	
Here we used the Leibniz rule together with the facts that $d(i_{-1}) = 0$ from \eqref{ddizero}, and $d(p_{-1} i_{-1}) = 0$ since $p_{-1} i_{-1}$ is the identity.

Now we check that $\Gamma_{I,d}$ has no cycles. But $d(i_{-1}) = 0$ so there are no oriented edges leaving the vertex $-1$. There will be an edge from $+1$ to $-1$, since $p_{-1}
d(i_{+1}) = \lambda \id_{B_s}$. But there will be no loop at $+1$, since $p_{+1} d(i_{+1}) = \lambda p_{+1} i_{-1} = 0$. In other words, after knowing \eqref{ddizero} and \eqref{pdinonzero}, the existence of a Fc-filtration follows from general principles!

Note that \eqref{pdinonzero} is analogous to the definiteness of the Lefschetz form in Hodge theory. Given morphisms $p_1, p_2 \in \Hom(B_s B_s, B_s)$ in degree $-k$, one can consider 
\begin{equation} \label{lefschetzform} p_1 \circ d^k(\overline{p_2}), \end{equation}
a degree zero endomorphism of $B_s$. Here $\overline{p_2}$ is the vertical flip of $p_2$, living in $\Hom(B_s, B_s B_s)$. By picking out the coefficient of the identity map, we get a bilinear pairing
\begin{equation} p_1 \circ d^k(\overline{p_2}) = \langle p_1, p_2 \rangle \cdot \id_{B_s}. \end{equation}
Then \eqref{pdinonzero} implies that the form in degree $-1$ is (positive or negative) definite. This is not surprising for a form on a one-dimensional space, but we use this as an example. In many similar decompositions, the corresponding forms appear to have a signature which matches the Hodge-Riemann bilinear relations.

This section was a preview of the conjectures in the sequel to \cite{EQsl2}, where we push this idea as far as we can take it.

\subsection{Implications in type $A$} \label{subsec-implications1}

Suppose we are in the setting of Proposition \eqref{prop:onecolorgood}, and that $R = \Bbbk[x_1, \ldots, x_n]$.

Let $s = (i,i+1)$. For the standard differential, $d(\alpha_s) = \alpha_s z_s$, where $z_s = x_i + x_{i+1}$. Since $z_s \ne 0$, we must have $g_s \ne 0$, so $g_s \notin R^s$. The
only nonzero linear polynomials $g$ which satisfy $d(g) = g^2$ are $g = x_i$ for some $i$. The only $g$ for which $d(g) = g^2$ and $g + s(g) = z_s$ are $g = x_i$ and $g = x_{i+1}$.

Let us briefly consider non-standard differentials on $R$, and ask whether we can find $g_s$, $z_s$, etcetera as above.

\begin{example} Consider the case when $n=2$ and the
differential satisfies \begin{equation} d(x_1) = x_2^2 - 2 x_1 x_2, \quad d(x_2) = x_1^2 - 2 x_1 x_2. \end{equation} Then \begin{equation} d(x_1 - x_2) = (x_1 - x_2)(-x_1 - x_2)
\end{equation} so that $z_s = -(x_1 + x_2)$. Meanwhile, if $\rho$ is a primitive third root of unity and $g = \rho x_1 + \rho^2 x_2$, the reader can verify that $d(g) = g^2$ and $g + s(g) = -(x_1 + x_2)$. \end{example}

\begin{example}
We have computed that the most general differential on $\Bbbk[x_1,x_2]$ which works. It is $S_2$-invariant and satisfies
\begin{equation} d(x_1) = A x_1^2 - 2C x_1 x_2 + C x_2^2. \end{equation}
The case $C = 0$ is the standard differential. Then $g = a_1 x_1 + a_2 x_2$ satisfies $d(g) = g^2$ and $g + s(g) = z_s$, when $a_1$ and $a_2$ are distinct roots of the quadratic equation \begin{equation} y^2 + y(C-A) + C(C-A) = 0. \end{equation}
\end{example}

\subsection{Relation to singular Soergel calculus}

A diagrammatic calculus for singular Soergel bimodules (also called the \emph{(diagrammatic) Hecke 2-category}) in type $A$ is long-standing work in progress of Elias-Williamson.
Diagrammatic calculus for dihedral groups is due to Elias \cite{ECathedral}, where one can also find a review of the background. We will not provide further review here. Let us examine
what kind of differentials could exist on the Hecke 2-category, and how they would restrict to the ordinary Hecke category. We will be brief and only provide summary results.

The one-color generators of the Hecke 2-category are oriented cups and caps, as below.
\[ \ig{1}{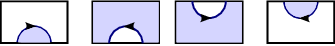} \]
They also span their morphism spaces up to the action of $R$ (in the white region). Consequently,
any differential on the Hecke 2-category must send these diagrams to a multiple of themselves by a linear polynomial; for the four diagrams pictured we call these linear polynomials $g_s$, $\bar{f}_s$, $f_s$, $\barg_s$ respectively. Checking that the isotopy/biadjunction relations are preserved by the differential will immediately imply that $f_s = - g_s$, $\bar{f}_s = - \barg_s$. This then implies \eqref{eq:denddot}, \eqref{eq:dstartdot},  \eqref{eq:dsplit} and \eqref{eq:dmerge}. Checking the remaining relations again gives the same constraints on $g_s$.

\begin{rem} This is a more restrictive and easier way to find a formula for \eqref{eq:dsplit}. In the computation of \S\ref{subsec-trivalent}, the fact that $f_1 = f_2 = 0$ is a
consequence of the differential being restricted from the Hecke 2-category, rather than being a consequence of the unit relation.

In similar fashion, the fact that $d$ preserves $R^s$ is forced upon any differential on the Hecke 2-category, rather than being a consequence. After all, $R^s$ is the endomorphism ring of the identity $1$-morphism of the parabolic subset $s$. \end{rem}

The defining idempotent decomposition is the decomposition of $R$ as an $(R^s,R^s)$-bimodule into two shifted copies of $R^s$. That this descends to a Fc-filtration gives exactly
the same two possibilities as in Proposition \ref{prop:onecolorgood}, although the computation is quite different.

\section{Two distant colors} \label{sec-distant}

\subsection{4-valent vertices} \label{subsec-fourvalent}

Let $s$ and $u$ be distant. Let us recall the assumption $(\star)$ from \S\ref{sec-prelims}, which states that all linear polynomials are in the span
of $R^s$ and $R^u$. One implication (with Demazure surjectivity) is the existence of a linear polynomial $\om_s$ for which $\pa_s(\om_s) = 1$ and $\pa_u(\om_s)=0$, and similarly
for a linear polynomial $\om_u$.

We now compute the differential applied to the generating $4$-valent vertex $\Xbg$, also known as the \emph{crossing}, which has degree $0$.

Once again, by examination of the double leaves basis, one can deduce that any degree $+2$ map $B_s \ot B_u \to B_u \ot B_s$ is a linear combination of taking $\Xbg$ and placing a linear polynomial in one of the four regions. Here is a useful consequence of $(\star)$ which we wish to record.

\begin{lem} \label{lem:shuffle4} Let $f$ be an arbitrary linear polynomial in $R$. Then
\begin{equation} \Xbgpoly{-f}{}{}{} + \Xbgpoly{}{f}{}{} + \Xbgpoly{}{}{-f}{} + \Xbgpoly{}{}{}{f} = 0. \end{equation} \end{lem}

\begin{proof} By $(\star)$ we can write $f = g + h$ where $g \in R^s$ and $h \in R^u$. Take the copies of $+f$ on the right and left, decompose them as $g + h$, and slide $g$ across the $s$-strand and $h$ across the $u$-strand. Now we have $-f + g + h = 0$ in both the top and bottom, so the result is zero. \end{proof}

\begin{rem} In fact, an analogous lemma applies to any $2m_{st}$-valent vertex: the alternating sum of placing $f$ in each region is zero. The proof is the same.\end{rem}

Now consider an arbitrary degree $+2$ morphism, which is a linear combination of $\Xbg$ with various linear polynomials. Using the lemma we can remove the polynomial from the leftmost region. Using the same simplification as in Remark \ref{rem:sliding}, we can assume that the polynomial on top is a multiple of $\om_s$ and the polynomial on bottom a multiple of $\om_u$, by modifying the polynomial on the right. Thus the degree $+2$ morphism $d(\Xbg)$ has the following form.
\begin{equation} 
d\left(~\Xbgpoly{}{}{}{}~\right) = \Xbgpoly{b \om_s}{}{}{} + \Xbgpoly{}{}{a \om_u}{} + \Xbgpoly{}{f}{}{}. \end{equation}
One could also slide these polynomials to the left, giving the equivalent statement
\begin{equation} 
d\left(~\Xbgpoly{}{}{}{}~\right) = \qquad \Xbgpoly{}{}{}{b \om_s + a \om_u\;\;\;} + \Xbgpoly{}{f}{}{}. \end{equation}

Let us check the relation which slides a dot through a crossing. There are actually four such relations, depending on where one puts the dot. First we place the dot in the upper left.
\begin{equation} \ig{.5}{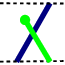} = \lineblue \finaldotgreen. \end{equation}
Taking the differential of both sides, we get
\begin{equation} \longpoly{g_u + b \om_s} \lineblue \finaldotgreen + \lineblue \finaldotgreen \longpoly{f + a \om_u} = \lineblue \finaldotgreen \longpoly{g_u}. \end{equation}
In order to break the blue strand with zero coefficient, we need $\pa_s(g_u + b \om_s) = 0$, or in other words
\begin{equation} b = - \pa_s(g_u). \end{equation}
Then, given that $g_u + b \om_s \in R^s$, the equality of both sides is equivalent to
\[ f + a \om_u + b \om_s + g_u = g_u, \] or in other words
\[ f + a \om_u + b \om_s = 0. \]

Similarly, checking the relation with a dot on the upper right gives
\begin{equation} a = \pa_u(g_s). \end{equation}
Thus we deduce that
\begin{equation} f = \pa_s(g_u) \om_s - \pa_u(g_s) \om_u. \end{equation}
This pins down the differential of the crossing exactly, given the known differential of the dots. Note that if $g_s \in R^u$ and $g_u \in R^s$ then $d(\Xbg) = 0$, and otherwise the differential is nonzero.

Let us simplify the answer. Adding and subtracting $g_u$ to the region on top, the polynomial in that region is
\begin{equation} - g_u + (g_u - \pa_s(g_u) \om_s). \end{equation}
Since $(g_u - \pa_s(g_u) \om_s)$ is $s$-invariant, it slides through to the rightmost region, leaving $-g_u$ behind. Similarly, we can add and subtract $g_s$ to the region on bottom, and slide $(\pa_u(g_s) \om_u - g_s)$ through the $u$-wall to the rightmost region. What remains is
\begin{equation} \label{eq:dXbg} 
d\left(~\Xbgpoly{}{}{}{}~\right) = \Xbgpoly{-g_u}{}{}{} + \Xbgpoly{}{}{g_s}{} + \Xbgpoly{}{g_u - g_s}{}{}. \end{equation}
Using Lemma \ref{lem:shuffle4} we can shuffle around these polynomials to obtain other nice descriptions of $d(\Xbg)$ as well. For example,
\begin{equation} 
d\left(~\Xbgpoly{}{}{}{}~\right) = \quad \Xbgpoly{}{}{}{-g_u} + \Xbgpoly{}{}{{\scriptscriptstyle \; g_s + g_u}}{} + \Xbgpoly{}{-g_s}{}{}\;\;\;. \end{equation}

Let us check the cyclicity of the $4$-valent vertex. We have
\begin{equation} 
d\left(~\ig{1}{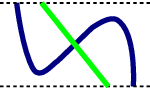}~\right) = \quad \Xgbpoly{-g_s}{}{}{} + \Xgbpoly{}{g_s - g_u}{}{} \quad + \Xgbpoly{}{}{*}{} + \quad \Xgbpoly{}{}{}{g_s + \barg_s}, \end{equation}
where $* = g_u - g_s - \barg_s$.
We used \eqref{eq:dXbg} to get four of these terms (which, instead of appearing on the top, right, and bottom as in \eqref{eq:dXbg}, now appear on the right, bottom, and left respectively because of the twisting). The remaining four terms came from taking the differential of the blue cap and cup. Now notice that $g_s + \barg_s = z_s \in R^u$, so it can be slid through the green strand. This cancels four of the terms, yielding
\begin{equation} 
d\left(~\ig{1}{Xbgtwist}~\right) = \quad \Xgbpoly{-g_s}{}{}{} + \Xgbpoly{}{g_s - g_u}{}{} \quad + \Xgbpoly{}{}{g_u}{}. \end{equation}
But this agrees with the formula \eqref{eq:dXbg} after swapping the colors $s$ and $u$, as desired.

We claim that checking the relations below is completely straightforward, and we leave it as an exercise to the reader. \begin{itemize}
\item Sliding a trivalent vertex through a crossing.
\item Crossings are inverse isomorphisms: $\Xbg \circ \Xgb = \linegreen \lineblue$.
\item The $A_1 \times A_1 \times A_1$ Zamolodchikov relation. \end{itemize}

\subsection{The commuting relation and its idempotent decomposition} \label{subsec-commuting}

One of the defining relations of the Hecke algebra is the commuting relation $b_s b_u = b_u b_s$, when $m_{su} = 2$. This is lifted by an isomorphism 
\begin{equation} B_s B_u \cong B_u B_s \end{equation}
in the Hecke category, and the inverse isomorphisms $p$ and $i$ are given by the 4-valent vertices. In fact, up to rescaling, these are the only choices of isomorphisms.

One should think that $B_s B_u$ has a decomposition with a single term, so that $I$ has size $1$. We still need to check that $\Gamma_{I,d}$ has no cycles, or in other words, that there is no loop at the single vertex. In other words, we need to check that $p d(i) = 0$ in order for the isomorphism to lift to a fantastic filtration, and for $[B_s][B_u] = [B_u][B_s]$ in the $p$-dg Grothendieck group.

\begin{prop} \label{prop:distantcolorgood} The isomorphism $B_s B_u \cong B_u B_s$ is an isomorphism of $p$-DG objects if and only if
\begin{equation} \Xbg \circ d\left(~\Xgb~\right) = 0, \end{equation}
if and only if $g_s \in R^u$ and $g_u \in R^s$, if and only if $d(\Xbg) = 0$. \end{prop}

\begin{proof} The first equivalence is definitional, and the remainder are very straightforward computations. \end{proof}
	
\subsection{Implications in type $A$} 

For the standard differential, as deduced in \S\ref{subsec-implications1}, when $s = s_i = (i,i+1)$ then either $g_{s_i} = x_i$ or $g_{s_i} = x_{i+1}$. In either case, $g_s \in R^u$ whenever $u$ is distant from $s$, as desired.

\subsection{Relation to thick calculus} \label{ssec:thickdistant}

There is very little additional perspective added from considering thick calculus (or singular calculus) for two distant colors. For completeness, we include a brief discussion.

In the thick calculus, there is an object $B_{s,u}$ corresponding to the parabolic subset $\{s,u\}$, which we draw with an olive strand. There are splitting and merging maps which give inverse isomorphisms between $B_{s,u}$ and the tensor product $B_s B_u$, as well as the tensor product $B_u B_s$. Here is the drawing of the isomorphism $B_{s,u} \to B_u B_s$.

\[ \ig{.75}{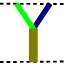} \]

The differential applied to this map will place a linear polynomial in each of the regions. Using $(\star)$, there is no need to put a polynomial in the region on top, so the polynomials must go on the left and right.
\begin{equation} d(\ig{.5}{thickcommsplit}) = \poly{f_1} \ig{.5}{thickcommsplit} + \ig{.5}{thickcommsplit} \poly{f_2}. \end{equation}
Moreover, by moving polynomials in $R^{su}$ from left to right, we can assume that $f_1 = b' \om_s + a' \om_u$ for some scalars $a', b'$.

The inverse isomorphism comes from flipping the diagram upside down, and applying the differential to the relation stating that they are inverse isomorphisms implies
\begin{equation} d(\ig{.5}{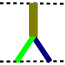}) = \poly{-f_1} \ig{.5}{thickcommmerge} + \ig{.5}{thickcommmerge} \poly{-f_2}. \end{equation}
For these inverse isomorphisms to give $p$-dg isomorphisms, it is immediate to compute that $f_1 = f_2 = 0$.

The isomorphism $B_{s,u} \to B_s B_u$ has a similar picture, with polynomials $h_1$ and $h_2$ instead. Composing the morphisms $B_s B_u \to B_{s,u} \to B_u B_s$ is supposed to give
the $4$-valent vertex, and the relations in the thick calculus are derived from this. So, without any appreciable difference from the computation in the rest of this section, we
deduce that 
\begin{equation} \label{fhfh} h_1 - f_1 = f_2 - h_2 = \pa_s(g_u)\om_s - \pa_u(g_s) \om_u. \end{equation}

Thus, unlike the ordinary diagrammatic Hecke category, a differential on the thick calculus is \emph{not} uniquely determined by what happens to the dots. One can choose $f_1$ and $f_2$ freely, and then $h_1$ and $h_2$ are determined by \eqref{fhfh}. However, if the differential is to be good (i.e. it satisfies Proposition \ref{prop:distantcolorgood}) then $f_1 = f_2 = h_1 = h_2 = 0$, and this additional flexibility disappears.

%
%
%

\section{Two adjacent colors} \label{sec-adjacent}

\subsection{Preliminaries} \label{sec-abusivenotation}

Let $s$ and $t$ be adjacent, so that $m_{st} = 3$. Our next task is to compute the differential applied to the generating $6$-valent vertex, which has degree $0$. So let us examine for
a time the space of degree $+2$ morphisms $B_s B_t B_s \to B_t B_s B_t$.

Inside this space we have six \emph{broken 6-valent vertices} which we call the 12 o'clock break, the 2 o'clock break, etcetera.
\[ \brokeXII \quad \brokeII \quad \brokeIV \quad \brokeVI \quad \brokeVIII \quad \brokeX \]
In fact, these six morphisms only span a four-dimensional subspace. There is a relation
\begin{equation} \label{aroundtheclockorig} \brokeX + \brokeXII = \brokeIV + \brokeVI, \end{equation}
together with its rotations around the clock. Using this, one can prove that 12, 2, 4, and 6 o'clock form a basis for this four-dimensional subspace.

Now consider the larger subspace spanned by the broken 6-valent vertices, and by morphisms obtained from the 6-valent vertex by adding a linear polynomial in some region. Forcing a polynomial from one region to another is possible by \eqref{eq:polyforce} at the cost of breaking some strands. Consequently, this subspace is spanned by the four broken 6-valent vertices above, and by the morphisms of the form
\[ \sixvalent \tallpoly{f}. \]

Meanwhile, there are five double leaves of degree $+2$, so the entire space of degree $+2$ morphisms is spanned by the previous subspace and one more morphism, which we can take to be
\[ \cupcapthing. \]

Thus we can assume that
\begin{equation} \label{eq:dsixorig} 
d\left(~\sixvalent~\right) = A \brokeXII + B \brokeVI + C \brokeII + D \brokeIV + E \cupcapthing + \sixvalent \tallpoly{f} \end{equation}
for some scalars $A, B, C, D, E$ and some linear polynomial $f$.

At this point, our computations are linear combinations of many very similar diagrams, and it helps to introduce some new and extremely abusive notation. First, instead of drawing a broken strand, we will merely draw a strand marked with a coefficient. Thus if $A$ is a scalar then
\begin{equation} 
\linebluelabel{A} := A \cdot \ig{1}{brokenblue}. 
\end{equation}
Second, with the understanding that we are only interested in morphisms of a particular degree, we \textbf{superimpose} diagrams rather than adding them together! 
So, if we knew we were discussing morphisms $B_s \to B_s$ of degree $2$, then
\[ \otherpoly{{\color{olive} f_1}} \linebluelabel{A} \otherpoly{{\color{olive} f_2}} \]
is actually shorthand for the sum
\[ \otherpoly{f_1} \linebluelabel{} + A \ig{1}{brokenblue} + \linebluelabel{} \otherpoly{f_2}. \]
This notation is horribly abusive (oh, if only our mothers could see us now!) but being able to draw a large linear combination succinctly has benefits both for the page count and for the readability and understandibility of this paper. For sanity, we will always use the {\color{olive} olive} color when we use this particular abuse of notation.

So, instead of \eqref{eq:dsixorig} we can write the very compact
\begin{equation} \label{eq:dsix} d(\sixvalentlabel{}{}{}{}{}{}{}{}) = \sixvalentlabel{A}{C}{D}{B}{}{}{f}{} + E \cupcapthing. \end{equation}

To give some more examples, here is a rewriting of \eqref{aroundtheclockorig}, after multiplication by $A$.
\begin{equation} \label{eq:aroundtheclock} \sixvalentlabel{A}{}{}{}{}{A}{}{} = \sixvalentlabel{}{}{A}{A}{}{}{}{}. \end{equation}
Here is a useful equation involving the trivalent vertex.
\begin{equation} \label{eq:aroundthetri} \mergelabel{}{}{A}{}{} = \mergelabel{A}{A}{}{-A\alpha_s}{} \quad. \end{equation}

Finally, before we begin the computation proper, let us set up notation for some important scalars. Let
\begin{equation} \kappa_{ss} = \pa_s(g_s), \quad \kappa_{st} = \pa_s(g_t), \quad \kappa_{ts} = \pa_t(g_s), \quad \kappa_{tt} = \pa_t(g_t), \end{equation}
\begin{equation} \bark_{ss} = \pa_s(\barg_s), \quad \bark_{st} = \pa_s(\barg_t), \quad \bark_{ts} = \pa_t(\barg_s), \quad \bark_{tt} = \pa_t(\barg_t). \end{equation}
Here are some things we know about these scalars. Since $g_s + \barg_s \in R^s$ and similarly for $t$, we have 
\begin{subequations} \label{kprops}
\begin{equation} \kappa_{ss} + \bark_{ss} = 0, \quad \kappa_{tt} + \bark_{tt} = 0.\end{equation}
Using \eqref{eq:zszt} we have
\begin{equation} \kappa_{st} + \bark_{st} + \kappa_{ts} + \bark_{ts} = 0. \end{equation}
\end{subequations}

\subsection{Pinning down the differential} 

We will pin down the coefficients $A, B, C, D, E, f$ by checking the ``death by pitchfork'' relations, which say that putting a dot on one input to a 6-valent vertex, and merging its two neighbors with a trivalent vertex, will yield the zero morphism. In this first example, we put the dot at 12 o'clock.
\begin{equation} \label{pitchforkontop} \ig{.5}{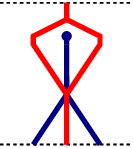} = 0 \end{equation}
Applying the differential to both sides we get
\begin{equation} 0 = {{\color{olive}
\labellist
\small\hair 2pt
 \pinlabel {$A$} [ ] at 32 40
 \pinlabel {$B$} [ ] at 32 12
 \pinlabel {$C$} [ ] at 40 37
 \pinlabel {$D$} [ ] at 40 13
 \pinlabel {$f$} [ ] at 49 25
 \pinlabel {$*$} [ ] at 40 52
\endlabellist
\centering
\ig{1}{pitchforkontop}
}} - E \ig{1}{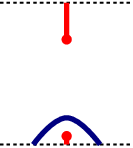}. \end{equation}
The polynomial $*$ is $g_s - \barg_t$. Now, the contributions of the terms with $B$, $D$, and $f$ are all zero, since they have a subdiagram with \eqref{pitchforkontop} inside. The $A$ term just creates a copy of $A \alpha_s$ which is added to $*$. Then $*$ can be forced to the right using \eqref{eq:polyforce}, and only the term which breaks the strand will survive. Consequently, the result is
\begin{equation} 0 = (C + \pa_t(g_s - \barg_t + A \alpha_s)) \ig{1}{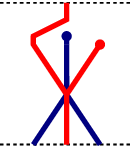} - E \ig{1}{Eterm}. \end{equation}
These two diagrams are linearly independent, from which we conclude that
\begin{equation} E = 0,\end{equation}
\begin{equation} \label{blah1} C - A + \kappa_{ts} - \bark_{tt} = 0. \end{equation}

Similarly, we can put the pitchfork on bottom, putting the dot on 6 o'clock. An entirely similar argument shows that ($E = 0$ again, and)
\begin{equation} \label{blah2} D - B + \bark_{st} - \kappa_{ss} = 0. \end{equation}

In general, the two-color Hecke category has an automorphism which flips a diagram upside-down and swaps the colors $s$ and $t$. This automorphism preserves the 6-valent vertex. The effect of this symmetry will be to swap $g_s$ with $\barg_t$, $\barg_s$ with $g_t$, $A$ with $B$, $C$ with $D$, $\kappa_{st}$ with $\bark_{ts}$, etcetera, while fixing $E$ and $f$. Thus this symmetry interchanges \eqref{blah1} with \eqref{blah2}.

Now we put the dot on 10 o'clock, and take the differential.
\begin{equation} 0 = {\color{olive} {
\labellist
\small\hair 2pt
 \pinlabel {$A$} [ ] at 32 42
 \pinlabel {$B$} [ ] at 32 11
 \pinlabel {$C$} [ ] at 40 37
 \pinlabel {$D$} [ ] at 40 13
 \pinlabel {$f$} [ ] at 47 25
 \pinlabel {$*$} [ ] at 16 31
 \pinlabel {$\barg_s$} [ ] at 7 8
\endlabellist
\centering
\ig{1}{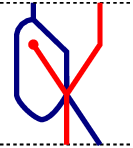}
}}. \end{equation}
This time the polynomial $*$ is $g_t - \barg_s - g_s$, with contributions coming from the red dot, the blue trivalent vertex, and the blue cup. Since $g_s + \barg_s = z_s \in R^s$, it can be slid out of this region. Thus the only surviving terms are
\begin{equation} 0 = (A + \kappa_{st}) \ig{1}{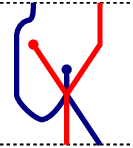}. \end{equation}
From this (and the corresponding computation for 8 o'clock, given by symmetry) we deduce that
\begin{equation} A = - \kappa_{st}, \quad B = - \bark_{ts}. \end{equation}
Then \eqref{blah1} and \eqref{blah2} give us
\begin{equation} \label{eq:CandDare} C = \bark_{tt} - \kappa_{ts} - \kappa_{st}, \quad D = \kappa_{ss} - \bark_{st} - \bark_{ts}. \end{equation}
	
Thus we have solved for $A, B, C, D, E$. If we resolved the 2 o'clock and 4 o'clock relations we would get equations for $\pa_s(f)$ and $\pa_t(f)$, but it is easier to check other
relations to determine $f$ precisely.

\subsection{Checking the relations} 

The relation called \emph{2-colored associativity} is the equality
\begin{equation} \label{eq:2assoc} \ig{1}{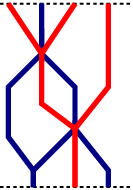} = \ig{1}{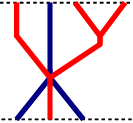}. \end{equation}

We apply the differential to the left hand side.
\begin{equation} 
d\left(~\ig{1}{2assoc1}~\right) = {\color{olive} {
\labellist
\small\hair 2pt
 \pinlabel {$A$} [ ] at 20 81
 \pinlabel {$A$} [ ] at 35 43
 \pinlabel {$B$} [ ] at 19 52
 \pinlabel {$B$} [ ] at 35 11
 \pinlabel {$C$} [ ] at 29 80
 \pinlabel {$C$} [ ] at 50 50
 \pinlabel {$D$} [ ] at 29 58
 \pinlabel {$D$} [ ] at 48 14
 \pinlabel {$f$} [ ] at 39 68
 \pinlabel {$f$} [ ] at 56 27
 \pinlabel {$-g_s$} [ ] at 16 29
\endlabellist
\centering
\ig{1}{2assoc1}
}} \end{equation}
Now we simplify this morphism. The central red strand, if broken, yields the zero morphism; consequently one of the $B$ terms does not contribute, and the $-g_s$ term can be forced across this red line to become $-t(g_s)$. This polynomial can be forced further across the blue line, breaking this line with a coefficient of $-\pa_s(t(g_s))$, which is added to the previous coefficient for breaking this line, namely $A+D$. Note that
\[ t(g_s) = g_s - \kappa_{ts} \alpha_t \]
so that
\begin{equation} \pa_s(t(g_s)) = \kappa_{ss} + \kappa_{ts}. \end{equation}
In particular, the line is broken with coefficient
\begin{equation} A + D - \pa_s(t(g_s)) = - \kappa_{st} + \kappa_{ss} - \bark_{st} - \bark_{ts} - \kappa_{ss} - \kappa_{ts} = 0 \end{equation}
where we used \eqref{kprops} to deduce that the coefficient was zero. Thus what remains at the end is
\begin{equation} 
d\left(~\ig{1}{2assoc1}~\right) = {\color{olive} {
\labellist
\small\hair 2pt
 \pinlabel {$A$} [ ] at 20 81
 \pinlabel {$B$} [ ] at 35 11
 \pinlabel {$C$} [ ] at 29 80
 \pinlabel {$C$} [ ] at 50 50
 \pinlabel {$D$} [ ] at 48 14
 \pinlabel {$*$} [ ] at 39 68
 \pinlabel {$f$} [ ] at 56 27
\endlabellist
\centering
\ig{1}{2assoc1}
}} \end{equation}
where the polynomial $*$ is $f - st(g_s)$.

Meanwhile, on the right hand side of 2-colored associativity, we have
\begin{equation} 
d\left(~\ig{1}{2assoc2}~\right) = {\color{olive} {
\labellist
\small\hair 2pt
 \pinlabel {$A$} [ ] at 24 39
 \pinlabel {$B$} [ ] at 24 8
 \pinlabel {$C$} [ ] at 35 30
 \pinlabel {$D$} [ ] at 33 10
 \pinlabel {$f$} [ ] at 53 20
 \pinlabel {$-g_t$} [ ] at 47 54
\endlabellist
\centering
\ig{1}{2assoc2}
} = {
\labellist
\small\hair 2pt
 \pinlabel {$A$} [ ] at 24 39
 \pinlabel {$B$} [ ] at 24 8
 \pinlabel {$C$} [ ] at 41 50
 \pinlabel {$C$} [ ] at 54 50
 \pinlabel {$D$} [ ] at 33 10
 \pinlabel {$f$} [ ] at 53 20
 \pinlabel {$*$} [ ] at 47 54
\endlabellist
\centering
\ig{1}{2assoc2}
}}. \end{equation}
Here $*$ represents $-g_t - C \alpha_t$, and the second equality followed from \eqref{eq:aroundthetri}.

Applying the differential to both sides of \eqref{eq:2assoc}, we have a linear combination of terms which look like either side of \eqref{eq:2assoc} except with one strand broken, and the coefficients on each strand match up perfectly. We also have $f$ in the rightmost region, matching perfectly, and a polynomial in the upper right region, which does not obviously match. So all that is required is for these polynomials to agree, namely,
\begin{equation} \label{blah7} f = - g_t - C \alpha_t + st(g_s). \end{equation}
This solves for the polynomial $f$.

To simplify further, note that
\[st(g_s) = g_s - \kappa_{ss} \alpha_s - \kappa_{ts} (\alpha_s + \alpha_t) \]
so that (using \eqref{eq:CandDare}) we have
\begin{equation} \label{blah8} f = g_s - g_t - (\kappa_{ss} + \kappa_{ts})\alpha_s - (\bark_{tt} - \kappa_{st}) \alpha_t. \end{equation}

There is another version of 2-colored associativity (which is not a rotation of this one), which can be checked by flipping upside-down and swapping colors, a symmetry we have previously discussed. This yields the equality
\begin{equation} f = \barg_t - \barg_s - (\bark_{tt} + \bark_{st}) \alpha_t - (\kappa_{ss} - \bark_{ts}) \alpha_s. \end{equation}
We must confirm that these two equations for $f$ are consistent. Taking the difference, we have
\[ g_s + \barg_s - g_t - \barg_t - (\kappa_{ts} + \bark_{ts}) \alpha_s + (\kappa_{st} + \bark_{st}) \alpha_t \]
which is equal to
\[ z_s - z_t - \pa_t(z_s) \alpha_s + \pa_s(z_t) \alpha_t.\]
Using \eqref{eq:zszt} and \eqref{eq:prop3}, this is zero as desired.

The next relation we should check is when a dot is placed on the 6-valent vertex. We can check what happens if a blue dot is placed, and determine what happens for a red dot by
symmetry. Confirming that the differential preserves this relation is extremely tedious but straightforward, and uses only the tricks already used above, so we leave it to the reader.

Finally, we should check the cyclicity relation. We postpone the 3-color relations until the next chapter.

First we must discuss the differential of the other 6-valent vertex, from $B_t B_s B_t \to B_s B_t B_s$. In the next chapter, the 6-valent vertex $B_s B_t B_s \to B_t B_s B_t$ will be denoted $\phi$, and the 6-valent vertex $B_t B_s B_t \to B_s B_t B_s$ will be denoted $\psi$.

We can
deduce many things about $\psi$ by applying the symmetry of the two-color Hecke category which swaps $s$ and $t$; this will swap $g_s$ with $g_t$, $\barg_s$ with $\barg_t$,
etcetera. Thus we have
\begin{equation} \label{eq:dothersix} 
d\left(~\othersixvalentlabel{}{}{}{}{}{}{}{}~\right) = \othersixvalentlabel{A'}{C'}{D'}{B'}{}{}{f'}{}, \end{equation}
where
\begin{subequations}
\begin{equation} A' = - \kappa_{ts}, \quad B' = - \bark_{st}, \end{equation}
\begin{equation} C' = \bark_{ss} - \kappa_{ts} - \kappa_{st}, \quad D' = \kappa_{tt} - \bark_{st} - \bark_{ts}, \end{equation}
\begin{equation} f' = g_t - g_s - (\kappa_{tt} + \kappa_{st})\alpha_t - (\bark_{ss} - \kappa_{ts}) \alpha_s. \end{equation}
\end{subequations}

Using \eqref{kprops} one can see that
\begin{equation} \label{eq:CDCprimeDprime} D' = -C, \quad C' = -D. \end{equation}

Now applying the formula \eqref{eq:dsix}, and taking the differential of cups and caps as well, we get
\begin{equation} 
d\left(~\ig{2}{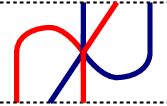}~\right) = {\color{olive}{
\labellist
\large\hair 2pt
 \pinlabel {$C$} [ ] at 32 40
 \pinlabel {$D$} [ ] at 41 38
 \pinlabel {$B$} [ ] at 40 12
 \pinlabel {$A$} [ ] at 22 38
\small\hair 2pt
 \pinlabel {$\barg_s$} [ ] at 55 24
 \pinlabel {$g_t$} [ ] at 7 24
 \pinlabel {$-\barg_t$} [ ] at 26 8
\tiny\hair 2pt
 \pinlabel {$f-g_s$} [ ] at 38 44
\endlabellist
\centering
\ig{2}{otherbig6valent}
}}. \end{equation}
Our goal will be to simplify this linear combination until it agrees with \eqref{eq:dothersix}, which we do by forcing all the polynomials to the rightmost region, one strand at a time.

Recall that $A = - \kappa_{st} = - \pa_s(g_t)$. If we try to force $g_t$ from the leftmost region across its neighboring blue strand (10 o'clock), we get $s(g_t)$ on the other side, plus a broken strand with scalar $+ \pa_s(g_t)$. Thus the 10 o'clock break has overall coefficient $A + \pa_s(g_t) = 0$. Similarly, if we force $-\barg_t$ across its neighboring blue strand (6 o'clock), we get $-s(\barg_t)$ on the other side, and break the strand with coefficient $-\pa_s(\barg_t) = -\bark_{st} = B'$. Combining these two manipulations we have
\begin{equation} d\left(~\ig{3}{6valentrot}~\right) = {\color{olive}{
\labellist
\large\hair 2pt
 \pinlabel {$C$} [ ] at 32 40
 \pinlabel {$D$} [ ] at 41 38
 \pinlabel {$B$} [ ] at 40 12
 \pinlabel {$B'$} [ ] at 32 10
\small\hair 2pt
 \pinlabel {$\barg_s$} [ ] at 55 24
\tiny\hair 2pt
 \pinlabel {$f-g_s$} [ ] at 38 44
 \pinlabel {$s(g_t)$} [ ] at 26 44
 \pinlabel {$-s(\barg_t)$} [ ] at 38 9
\endlabellist
\centering
\ig{3}{otherbig6valent}
}}. \end{equation}

In similar fashion, we continue to force more polynomials to the right. Forcing $s(g_t)$ across its red neighbor will produce $ts(g_t)$ on the other side, and break the strand with coefficient $\pa_t(s(g_t)) = \kappa_{tt} + \kappa_{st}$. Adding this to $C$, the overall coefficient on the 12 o'clock break will be
\[ C + \kappa_{tt} + \kappa_{st} = \bark_{tt} + \kappa_{tt} - \kappa_{ts} = - \kappa_{ts} = A'. \]

To get the coefficient of the 2 o'clock break, we need to force $f - g_s + ts(g_t)$ across its neighboring blue strand. In the same way that we deduced \eqref{blah8} from \eqref{blah7}, one can also deduce that
\begin{equation} f = g_s + C' \alpha_s - ts(g_t). \end{equation}
Hence $f - g_s + ts(g_t) = C' \alpha_s$, and it breaks the 2 o'clock strand with coefficient $\pa_s(C' \alpha_s) = 2 C'$. Adding this to the existing coefficient $D = -C'$, we get an overall coefficient of $C'$.

The ultimate coefficient of the 4 o'clock break will be
\[ B - \pa_t(s(\barg_t)) = - \bark_{ts} - \bark_{tt} - \bark_{st} = D'.\]

Putting this together, we get
\begin{equation} 
d\left(~\ig{1.5}{6valentrot}~\right) = {\color{olive}{
\labellist
\large\hair 2pt
 \pinlabel {$A'$} [ ] at 32 40
 \pinlabel {$C'$} [ ] at 42 38
 \pinlabel {$D'$} [ ] at 40 12
 \pinlabel {$B'$} [ ] at 32 10
\small\hair 2pt
 \pinlabel {$*$} [ ] at 55 24
\endlabellist
\centering
\ig{1.5}{otherbig6valent}
}} \end{equation}
where
\begin{equation}\label{eq:fprimestar} * = \barg_s + s(f - g_s + ts(g_t)) - ts(\barg_t) = \barg_s - C' \alpha_s - ts(\barg_t). \end{equation}
It remains to show that $* = f'$. We leave this to the reader, having done enough similar computations.

\subsection{The braid relation and its defining idempotent decomposition: abstractions}
\label{subsec-mixedFc}

The braid relation on the Hecke algebra, reinterpreted in the Kazhdan-Lusztig presentation, is
\begin{equation} \label{eq:KLbraid} b_s b_t b_s - b_s = b_t b_s b_t - b_t, \end{equation}
as both sides are actually descriptions of the Kazhdan-Lusztig basis element $b_{sts}$. The categorification of this statement is the direct sum decompositions
\begin{equation} \label{eq:KLbraidcatfd} B_s B_t B_s \cong M \oplus B_s, \quad B_t B_s B_t \cong M' \oplus B_t, \end{equation}
together with the isomorphism
\begin{equation} \label{eq:MM} M \cong M'. \end{equation}
However, $M$ and $M'$ are not objects in the diagrammatic Hecke category, but are only objects in the Karoubi envelope, being the image of certain idempotents.

In the introduction we discussed the practical way to prove a direct sum decomposition, see \eqref{directsumdecomp}. Analogously, one wishes to prove \eqref{eq:KLbraidcatfd} and \eqref{eq:MM} practically, but using only morphisms between Bott-Samelson bimodules. For sake of brevity let us write $X = B_s B_t B_s$ and $Y = B_t B_s B_t$. One should provide morphisms
\[ i_s \co B_s \to X, \quad i_{t} \co B_t \to Y, \quad p_s \co X \to B_s, \quad p_t \co Y \to B_t, \]
\[ \phi \co X \to Y, \quad \psi \co Y \to X, \]
satisfying
\begin{subequations} \label{subeq:KLbraidcatfy}
\begin{equation} p_s i_s = \id_s, \quad p_t i_t = \id_t, \end{equation}
\begin{equation} \phi i_s = 0, \quad \psi i_t = 0, \quad p_s \psi = 0, \quad p_t \phi = 0, \end{equation}
\begin{equation} \id_X = \psi \phi + i_s p_s, \quad \id_Y = \phi \psi + i_t p_t. \end{equation}
\end{subequations}
One should think that the maps $\phi$ and $\psi$ pass through the common summand $M \cong M'$, so their composition is the idempotent projecting to this summand. In particular, $M$ is to be identified with the object $(X, \psi \phi)$ in the Karoubi envelope, and $M'$ with $(Y, \phi \psi)$. The map $\phi$ induces an isomorphism
\[ \bar{\phi} \co (X,\psi \phi) \to (Y, \phi \psi)\]
and similarly $\psi$ induces the inverse isomorphism $\bar{\psi}$. Now we ask what extra conditions produce the appropriate relations on the Grothendieck group in the $p$-dg setting.

\begin{rem} Everything we say in this section, including the main result (Proposition \ref{prop:mixedFc}), is easily adaptable to the general situation where one has
\[ X \cong M \oplus P, \quad Y \cong M' \oplus Q, \quad M \cong M' \]
in some $p$-dg category, where $P$ and $Q$ are genuine objects, while $M$ and $M'$ are only objects in the Karoubi envelope. \end{rem}

The identity of $X$ is decomposed as a sum of two orthogonal idempotents $e_1 = i_s p_s$ and $e_2 = \psi \phi$. In order for the decomposition $X \cong M \oplus B_s$ to be a
dg-filtration, we need\footnote{As a reminder, this condition will imply that either $\Hom(X,-) e_1$ or $\Hom(X,-) e_2$ is preserved by the differential, as a left $p$-dg module over the
category, and that the other one is the quotient of $\Hom(X,-)$ by the first. It also implies the analogous condition for right $p$-dg modules.} either $d(e_1) e_2 = 0$ or $d(e_2) e_1 = 0$. Now
\begin{equation} d(e_1) e_2 = d(i_s p_s) \psi \phi = i_s d(p_s) \psi \phi \end{equation}
since $p_s \psi = 0$. Clearly $d(e_1) e_2=0$ if $d(p_s) \psi \phi = 0$. Conversely, by postcomposing with $p_s$, if $d(e_1) e_2 = 0$ then $d(p_s) \psi \phi = 0$. Using a similar argument, precomposing with either $\psi$ or $\phi$, we see that
\begin{equation} d(e_1) e_2 = 0 \iff d(p_s) \psi = 0. \end{equation}
We will use this pre- and post-composition trick several times below.

Similarly,
\begin{equation} d(e_2) e_1 = 0 \iff d(\phi) i_s = 0. \end{equation}
This seems entirely analogous, but the proof is slightly trickier. Clearly
\begin{equation} d(e_2) e_1 = \psi d(\phi) i_s p_s \end{equation}
so
\begin{equation} d(e_2) e_1 = 0 \iff \psi d(\phi) i_s = 0 \iff \phi \psi d(\phi) i_s = 0.\end{equation}
But in fact $\phi \psi$ can be replaced with $\id_{tst}$ here, since the difference is $i_t p_t d(\phi) i_s$, and
\begin{equation} p_t d(\phi) i_s = d(p_t \phi) i_s - d(p_t) \phi i_s = 0. \end{equation}	
	
In addition to checking that the decomposition on $X$ is a dg-filtration, we need to confirm that the image of $i_s p_s$ has the appropriate $p$-dg structure, which amounts to checking that
\begin{equation} d(p_s) i_s = 0 \end{equation}
as a degree $2$ endomorphism of $B_s$.

Analogous conditions need to hold for the decomposition of $Y$.

Finally, we need to confirm that $M = (X,\psi \phi)$ and $M' = (Y,\phi \psi)$ are isomorphic as $p$-dg modules, with their induced differentials. Recall that, if $(X,e)$ and $(Y,f)$ are two objects in the Karoubi envelope, then the induced differential $\bar{d}$ on $\Hom((X,e),(Y,f)) = f \Hom(X,Y) e$ is
\begin{equation} \bar{d}(f \xi e) = f d(f \xi e) e \end{equation}
where $\xi \in \Hom(X,Y)$ and $d$ is the usual differential on $\Hom(X,Y)$. If $\bar{\phi}$ and $\bar{\psi}$ are to induce inverse isomorphisms of $p$-dg modules then we need
\begin{equation} \bar{d}(\bar{\phi}) \bar{\psi} = 0 \end{equation}
and (equivalently)
\begin{equation} \bar{d}(\bar{\psi}) \bar{\phi} = 0. \end{equation}
The first equation unravels to
\[ \phi \psi d(\phi) \psi \phi = 0 \]
which is equivalent by pre- and post-composition to
\begin{equation} \label{foo11a} \psi d(\phi) \psi = 0. \end{equation}
The second equation is equivalent to
\begin{equation} \label{foo11b}\phi d(\psi) \phi = 0. \end{equation}
	
To convince the reader that \eqref{foo11a} and \eqref{foo11b} are equivalent, let us pre- and post-compose \eqref{foo11a} with $\phi$. The result is
\[ \phi \psi d(\phi) \psi \phi = \phi \psi d(\phi \psi) \phi - \phi \psi \phi d(\psi) \phi = \phi \psi d(\phi \psi) \phi - \phi d(\psi) \phi. \]
So $\psi d(\phi) \psi = 0$ will imply that $\phi d(\psi) \phi = 0$ so long as $\phi \psi d(\phi \psi) \phi = 0$. But since $d(\id_{tst}) = 0$, we know that $d(\phi \psi) = - d(i_t p_t)$. Then
\[ \psi d(\phi \psi) \phi = - \psi d(i_t) p_t \phi - \psi i_t d(p_t) \phi = 0 + 0 = 0, \] as desired.

Together, all these conditions imply that $X$ is fantastically filtered by $M$ and $B_s$ in the Karoubi envelope, they $Y$ is fantastically filtered by $M'$ and $B_t$, and that $M \cong M'$ as $p$-dg objects. In conclusion, we have proven the following result.

\begin{prop} \label{prop:mixedFc} Given maps $i_s, p_s, i_t, p_t, \phi, \psi$ as in \eqref{subeq:KLbraidcatfy}, then consider the following graph.
	
\begin{equation}
\begin{tikzcd}[column sep=huge]
 B_s \ar[loop left, "d(p_s)i_s"] \ar[r, "d(\phi)i_s", yshift=.7ex] & M' \ar[l, "d(p_s) \psi", yshift=-.7ex] \ar[r, "\psi d(\phi) \psi", yshift=.7ex] & M \ar[l, "\phi d(\psi) \phi", yshift=-.7ex] \ar[r, "d(p_t) \phi", yshift=.7ex] & B_t \ar[l, "d(\psi) i_t", yshift=-.7ex] \ar[loop right, "d(p_t)i_t"]
\end{tikzcd}
\end{equation}
Erase the edges with zero labels (noting that if one edge between $M$ and $M'$ is zero, then so is the other). If the resulting graph has no cycles or loops, then the idempotent
decompositions \eqref{eq:KLbraidcatfd} and \eqref{eq:MM} lift to fantastic filtrations, the objects $M$ and $M'$ are cofibrant, and the braid relation \eqref{eq:KLbraid} holds in the $p$-dg Grothendieck group. \end{prop}

\begin{rem} The existence of a fantastic filtration is useful not only because it implies a relation in the Grothendieck group, but also because it implies, e.g., that $B_s B_t
B_s$ is in the triangulated hull of $B_s$ and $M$, and hence in the triangulated hull of $\{B_s, B_t, B_t B_s B_t\}$. Since all the other objects are cofibrant, the 2/3 rule in triangulated categories implies that $M$ is cofibrant. \end{rem}

\begin{rem} One should think of the pair of edges between $M$ and $M'$ as being like a loop. There is a partial idempotent completion where we could add a new object $Z$
corresponding to either $M$ or $M'$, and rewrite the idempotent decompositions as \[ X \cong B_s \oplus Z, \quad Y \cong B_t \oplus Z. \] Then we could use the ordinary theory of
Fc-filtrations to analyze these decompositions. We would get two graphs corresponding to the right and left halves of the graph in Proposition \ref{prop:mixedFc}, and in each graph
there would be a loop at $Z$ which corresponds to the edges between $M$ and $M'$. \end{rem}

\subsection{The braid relation and its defining idempotent decomposition: computations}
\label{subsec-braidtime}

Now let us compute the graph of Proposition \ref{prop:mixedFc}. We should note that all maps $i_s, p_s, \phi$, etcetera are uniquely determined up to scalar, living in one-dimensional
Hom spaces, so there is only one graph to compute. The maps $\phi$ and $\psi$ are 6-valent vertices, and the remaining maps are pitchforks.

First we check the loop at $B_s$. We have
\begin{equation} \label{eq:loopatBs} d(p_s) i_s = \needlepoly{*}, \end{equation}
where $*$ is the polynomial
$\alpha_t(g_t - \barg_s)$. Note that $p_s i_s$ is the same picture with a red barbell $\alpha_t$ inside. Applying $d$ to $p_s$ produces the extra factor of $g_t - \barg_s$.

For \eqref{eq:loopatBs} to be zero we need $*$ to be $s$-invariant, which requires that $g_t - \barg_s$ is proportional to $s(\alpha_t) = \alpha_s + \alpha_t$. This is a new condition! The scalar of proportionality can be determined by applying Demazure operators. Since $\pa_s(s(\alpha_t)) = \pa_t(s(\alpha_t)) = 1$, we see that
\begin{equation} \label{eq:foo12a} g_t - \barg_s = (\kappa_{st} - \bark_{ss}) (\alpha_s + \alpha_t) = (\kappa_{tt} - \bark_{ts}) (\alpha_s + \alpha_t). \end{equation}

Checking the loop at $B_t$, we get the analogous equation
\begin{equation} \label{eq:foo12b} g_s - \barg_t = (\kappa_{ts} - \bark_{tt}) (\alpha_s + \alpha_t) = (\kappa_{ss} - \bark_{st}) (\alpha_s + \alpha_t). \end{equation}

Now we check the pair of edges between $M$ and $M'$. We have
\begin{equation} \psi d(\phi) \psi = {{\color{olive}
\labellist
\small\hair 2pt
 \pinlabel {$A$} [ ] at 32 95
 \pinlabel {$B$} [ ] at 31 52
 \pinlabel {$C$} [ ] at 48 97
 \pinlabel {$D$} [ ] at 47 50
 \pinlabel {$f$} [ ] at 56 75
\endlabellist
\centering
\ig{1}{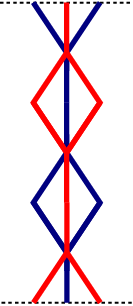}
}}.\end{equation}
The $A$ term and the $B$ term vanish. We encourage the reader to confirm that the $C$ term, the $D$ term, and the $f$ term are linearly independent\footnote{After all, $M$ is supposed to be isomorphic to the indecomposable Soergel bimodule $B_{sts}$. By the Soergel hom formula $\End^2(B_{sts})$ is spanned by: linear polynomials times the identity (the $f$ term), a degree-$2$ map factoring through $B_{st}$ (the $D$ term), and a degree $2$ map factoring through $B_{ts}$ (the $C$ term).}! Thus
\begin{equation} \psi d(\phi) \psi = 0 \iff C = D = f = 0. \end{equation}
Using \eqref{eq:CDCprimeDprime}, this will imply that $C' = D' = 0$ as well (and also $f'=0$).

Note that $f = -g_t - C \alpha_t + st(g_s)$ by \eqref{blah7}, so if $f = C = 0$ then
\begin{equation} g_t = st(g_s). \end{equation}
If $g_s = 0$ then $g_t = 0$ and vice versa. Recall that Proposition \ref{prop:onecolorgood} left us with two cases for each color. Either $g_s = \barg_s = 0$ or $g_s \ne \barg_s = s(g_s) \ne 0$, and similarly for $t$. We can now observe that we are in the same case for both $s$ and $t$: either $g_s = g_t = 0$, or $g_s \ne 0 \ne g_t$. In the case where $g_s = g_t = 0$ one deduces that the differential is also zero on both 6-valent vertices, so it is zero on any diagram with colors in $\{s,t\}$. This is a valid solution, but one which makes the computations trivial.

Let us restrict our attention to the other case. Thus we assume that
\begin{equation} \label{eq:assumesofar} \barg_s = s(g_s), \quad \barg_t = t(g_t), \quad g_t = st(g_s). \end{equation}
We leave it as an exercise to verify that \eqref{eq:assumesofar} implies that $C = D = f = f' = 0$ as well as \eqref{eq:foo12a} and \eqref{eq:foo12b}. Thus \eqref{eq:assumesofar} is the only assumption we need. The reader attempting this exercise will be helped by the formula
\begin{equation} \label{eq:funnybraid} st \pa_s(h) = \pa_t(st(h)) \end{equation}
for any polynomial $h$.

Recall also the scalar $\kappa_s = \pa_s(g_s)$ which we used in a previous chapter. Then $g_t = st(g_s)$ and \eqref{eq:funnybraid} imply that
\begin{equation} \kappa_s = \pa_s(g_s) = \pa_t(g_t) = \kappa_t. \end{equation}
We write 
\begin{equation} \label{eq:truekappadef} \kappa := \kappa_s = \kappa_t. \end{equation}

\begin{rem} If $g_t = st(g_s)$ then $\barg_t = tst(g_s)$. So it makes sense that $g_s - \barg_t = g_s - tst(g_s)$ is proportional to $\alpha_s + \alpha_t$, since $tst$ is a reflection and $\alpha_s + \alpha_t$ is its root. \end{rem}

Now we analyze the two remaining pairs of edges, under the requirements that the loops vanished above. We have
\begin{equation} d(p_t) \phi = {{\color{olive}
\labellist
\small\hair 2pt
 \pinlabel {$*$} [ ] at 40 52
\endlabellist
\centering
\ig{1}{pitchforkontop}
}} \end{equation}
where $*$ is the polynomial $g_s - \barg_t = g_s - tst(g_s)$. This polynomial can be forced out, and only the term which breaks the pitchfork will survive, so 
\begin{equation} d(p_t) \phi = 0 \iff \pa_t(g_s - \barg_t) = 0. \end{equation}
However, we have already seen that $g_s - \barg_t$ is proportional to $\alpha_s + \alpha_t$, which is not $t$-invariant. Thus $\pa_t(g_s - \barg_t) = 0$ if and only if $g_s - \barg_t = 0$! Consequently,
\begin{equation} \label{maybethis} d(p_t) \phi = 0 \iff g_s = \barg_t= sts(g_s). \end{equation}
Meanwhile,
\begin{equation} d(\psi) i_t = { {\color{olive}
\labellist
\small\hair 2pt
 \pinlabel {$A'$} [ ] at 31 60
 \pinlabel {$B'$} [ ] at 32 29
\endlabellist
\centering
\ig{1}{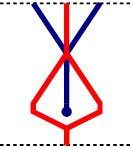}
}} \end{equation}
Only the $B'$ term will contribute, and it will contribute $\pa_t(B' \alpha_s) = - B'$. Recall that $B' = - \pa_s(\barg_t)$. Using \eqref{eq:funnybraid}, we can also observe that
\begin{equation} - \pa_s(\barg_t) = \pa_t(g_s). \end{equation} Then
\begin{equation} \label{maybethat} d(\psi) i_t = 0 \iff B' = 0 \iff \pa_s(\barg_t) = 0 \iff \pa_t(g_s) = 0. \end{equation}

Only one of these two conditions \eqref{maybethis} and \eqref{maybethat} need hold, so either $g_s = \barg_t = sts(g_s)$ or $B' = 0 = \pa_s(\barg_t) = \pa_t(g_s)$, but not necessarily both. That is, $g_s$ is either fixed by the reflection $t$ or by the reflection $sts$. If both hold, then $g_s = 0$.

Similarly, either $d(p_s) \psi = 0$ or $d(\phi) i_s = 0$, from which we deduce that either $g_t = \barg_s = sts(g_t)$ or $B = 0 = \pa_t(\barg_s) = \pa_s(g_t)$. These two conditions are equivalent to the conditions above, but in the reverse order: if $g_s$ is fixed by $sts$ then $g_t = st(g_s)$ is fixed by $s$, and if $g_s$ is fixed by $t$ then $g_t$ is fixed by $sts$.

Let us consider one possibility, where $B = 0$, so that $g_t$ is fixed by $s$ and $g_s$ is fixed by $sts$. In this case both $A = B = 0$, so that the differential kills $\phi$! We leave the reader to determine that $B' = - A'  = \pa_t(g_s) \ne 0$. Moreover, 
\begin{equation} \pa_t(g_s) = \pa_t(tst(g_s)) = - \pa_t(st(g_s)) = - \pa_t(g_t) = - \kappa. \end{equation} Thus, $A' = \kappa$ and $B' = - \kappa$ and our differential satisfies
\begin{equation} \label{eq:possibility1} 
d\left(~\sixvalent~\right) = 0, \quad d\left(~\othersixvalentlabel{}{}{}{}{}{}{}{}~\right) = \othersixvalentlabel{\kappa}{}{}{-\kappa}{}{}{}{}. \end{equation}
	
Analogously, the other possibility, where $B' = 0$, also satisfies $A' = 0$ and $B = -A = \pa_s(g_t)$, and
\begin{equation} \pa_s(g_t) = \pa_s(st(g_s)) = \pa_s(s(g_s)) = - \kappa. \end{equation} so that $A = \kappa$ and $B = -\kappa$ and
\begin{equation} \label{eq:possibility2} 
d\left(~\ig{.5}{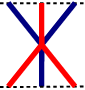}~\right) = 0, \quad d\left(~\sixvalentlabel{}{}{}{}{}{}{}{}~\right) = \sixvalentlabel{\kappa}{}{}{-\kappa}{}{}{}{}. \end{equation}
These are our two options.

Let us summarize.

\begin{prop} \label{prop:adjcolorgood} The graph from Proposition \ref{prop:mixedFc} has no loops or cycles if and only if one of the following three possibilities holds.
\begin{enumerate} \item $g_s = g_t = 0$, and the differential is zero on every diagram with colors in $\{s,t\}$.
	\item The equations \eqref{eq:assumesofar} hold, and $g_t$ is fixed by $s$, and $g_s = t(g_t)$ is fixed by $sts$. Then the differential on 6-valent vertices is given by \eqref{eq:possibility1}, where $\kappa = \pa_s(g_s) = \pa_t(g_t)$.
	\item The equations \eqref{eq:assumesofar} hold, and $g_s$ is fixed by $t$, and $g_t = s(g_s)$ is fixed by $sts$. Then the differential on 6-valent vertices is given by \eqref{eq:possibility2}, where $\kappa = \pa_s(g_s) = \pa_t(g_t)$.
\end{enumerate}
In other words, any good differential must satisfy one of these three possibilities.
\end{prop}

\begin{example} For the standard differential on $\Bbbk[x_1, \ldots, x_n]$, we deduced in \S\ref{subsec-implications1} that either $g_i = x_i$ or $g_i = x_{i+1}$. Ere now it seemed
we could make this choice for each $i$ independently. However, Proposition \ref{prop:adjcolorgood} forces the choices of $g_i$ and the choice of $g_{i+1}$ to be related to each
other, since $g_{i+1} = s_i s_{i+1} g_i$. If $g_i = x_i$ for some $i$ then $g_{i+1} = x_{i+1}$, and consequently $g_j = x_j$ for all $j$. If $g_i = x_{i+1}$ then $g_{i+1} = x_{i+2}$,
and consequently $g_j = x_{j+1}$ for all $j$. \end{example}

\subsection{Implications in simply laced types} \label{subsec-orientations}

Let $(W,S)$ be an irreducible Coxeter group with $m_{st} \in \{2,3\}$ for all $s, t \in S$.

If $g_s = 0$ for any $s \in S$, then by Proposition \ref{prop:adjcolorgood} we must have $g_t = 0$ for all $t$ in the same connected component of the Coxeter graph as $s$. This is the
boring case.

Otherwise, Proposition \ref{prop:adjcolorgood} gives one a dichotomy for each pair $s, t \in S$ with $m_{st} = 3$. We will encode this with an orientation on the edges in the Coxeter
graph: the edge points from $s$ to $t$ if $g_s$ is fixed by $t$.

Now suppose that $s, t, u \in S$ generate a copy of $A_3$ inside $W$, with $m_{su} = 2$. Suppose that the $s-t$ edge is oriented from $s$ to $t$. Then $g_s$ is fixed by $t$. We also know that $g_s$ is
fixed by $u$, thanks to Proposition \ref{prop:distantcolorgood}. Now, $g_t = st(g_s) = s(g_s)$, so that \[u(g_t) = us(g_s) = su(g_s) = s(g_s) = g_t.\] Consequently $g_t$ is fixed by
$u$, and the $t-u$ edge is oriented from $t$ to $u$. Similarly, we leave the reader to deduce that when the $s-t$ edge is oriented from $t$ to $s$, the $t-u$ edge must be oriented from
$u$ to $t$. In particular, this implies that every copy of $A_3$ must be consistently oriented: $t$ is never a source or a sink. This fact is sad because of the following proposition.

\begin{prop} There is no good differential in any simply laced type outside of finite or affine type $A$. \end{prop}
	
\begin{proof} Outside of finite and affine type $A$, the Coxeter graph must have a subgraph of type $D_4$. There is no way to orient $D_4$ such that every copy of $A_3$ inside is consistent. \end{proof}

\subsection{Implications in type $A$} 

We continue to let $g_i$ denote $g_{s_i}$ when $s_i = (i,i+1)$, for $1 \le i \le n-1$. As just noted in \S\ref{subsec-orientations}, any nonzero good differential must induce a
consistent orientation on the Dynkin diagram.

Let us begin by quickly discussing the case where the realization is the standard one, with polynomial ring $\Bbbk[x_1, \ldots, x_n]$, and the differential is assumed to be the standard one, with
$d(x_i) = x_i^2$. We have proven that there are only two nonzero good differentials possible which extend the standard differential on the polynomial ring, the one where $g_i = x_i$ for all
$i$, and the one where $g_i = x_{i+1}$ for all $i$. Technically we still need to check the three-color relations, but this is done in the next chapter. We refer to the differential where
$g_i = x_i$ for all $i$ as having the \emph{standard orientation}, and the other differential as having the \emph{reverse orientation}. In Manin-Schechtman theory, braid relations $s_i
s_{i+1} s_i \to s_{i+1} s_i s_{i+1}$ can be given the lexicographic orientation, and the set of all reduced expressions for a given element becomes a semioriented graph (the commuting
braid relations $s_i s_j = s_j s_i$ have no orientation). This graph has a unique source and a unique sink, up to commuting braid relations. See \cite{EThick} for more details. For the
standardly oriented good differential, $d$ will kill any 6-valent vertex corresponding to a reverse-oriented braid relation. Vice versa, for the reverse oriented good differential,
6-valent vertices corresponding to standardly oriented braid moves are killed. Let us refer to the $p$-dg category associated to the standard orientation as the \emph{standard $p$-dg diagrammatic Hecke category}.

Now we work with an arbitrary realization and its polynomial ring $R$, but under the assumption that its diagrammatic Hecke category possesses a good differential. For ease of discussion
we will assume the Dynkin diagram gets the standard orientation; the other case can be handled similarly, or by duality. We will prove that our $p$-dg category is isomorphic to the
standard $p$-dg diagrammatic Hecke category, perhaps after extension of the realization.

We know that $g_1 \in R^{s_2}$. We also know that $g_1 \in R^{s_j}$ for all $j \ge 3$, by Proposition \ref{prop:distantcolorgood}. Thus $g_1$ is invariant in all simple reflections but
the first, so it is invariant in the parabolic subgroup $S_1 \times S_{n-1}$. In other words, $g_1$ has the same stabilizer in the symmetric group as does $x_1$ in the standard
polynomial ring. Similarly, $g_i$ has the same stabilizer as $x_i$ for all $1 \le i \le n-1$, which one can check in the same way, or can confirm since $g_i = s_{i-1} s_{i-2} \cdots s_2
s_1(g_1)$. Let us denote $s_{n-1}(g_{n-1})$ by $g_n$; it has the same stabilizer as $x_n$.

Let $\kappa = \pa_1(g_1)$, which is also equal to $\pa_i(g_i)$ for all $i$ by \eqref{eq:funnybraid}. Note that, by the usual definition of Demazure operators, for all $1 \le i \le n-1$
one has \begin{equation} g_i - g_{i+1} = g_i - s_i(g_i) = \kappa \alpha_i. \end{equation} In particular, the span of the $g_i$ contains all the simple roots, and has dimension at least
$n-1$. It is also preserved by the differential, since $d(g_i) = g_i^2$. Thus we may as well restrict our realization and our attention to the span of $\{g_i\}_{i=1}^n$.

\begin{prop} There is a ring homomorphism from $\Bbbk[x_1, \ldots, x_n]$ to the subring of $R$ generated by $g_1, \ldots, g_n$, sending $x_i \mapsto g_i$. This homomorphism is
$S_n$-equivariant and intertwines the standard differential with the differential on $R$. In fact, it is an isomorphism so long as the characteristic of $\Bbbk$ is not $2$. If the characteristic is $2$, it is either an isomorphism or is the quotient map by the ideal generated by $e_1 = x_1 + \ldots + x_n$. \end{prop}

\begin{proof} The first two sentences are straightforward, so it remains to show that the homomorphism is an isomorphism. Let $I$ be the kernel of the map. If $g_1, \ldots, g_n$
is linearly independent, then they must be algebraically independent (since they are linear terms inside a polynomial ring), and $I = 0$.

Suppose to the contrary that there were a linear dependence $\sum a_i g_i = 0$ for some scalars $a_i$. Let $y = \sum_{i \ge 3} a_i g_i$, so that $a_1 g_1 + a_2 g_2 + y = 0$.
Since $y$ is fixed by $s_1$, so must be $a_1 g_1 + a_2 g_2$. Since $g_2 = s_1(g_1) \ne g_1$, this implies that $a_1 = a_2$. By similar arguments, $a_2 = a_3$, etcetera, and all
the scalars $a_i$ are equal. So the only possible linear dependence relation is $a(g_1 + \cdots + g_n)$ for some nonzero scalar $a$. Since the linear polynomials form a free $\Bbbk$-module in any realization, a relation of the form $a(g_1 + \cdots + g_n) = 0$ for $a \ne 0$ implies that $g_1 + \cdots + g_n = 0$.

Thus if $I$ is nonzero then $I$ contains the symmetric polynomial $e_1 = x_1 + \ldots + x_n$. Since there is at most one linear relation, at least $n-1$ of the elements $g_i$ are
algebraically independent. If $I$ is any bigger than the ideal generated by $e_1$ then this contradicts the algebraic independence of $\{g_1, \ldots, g_{n-1}\}$.

Now $I$ is preserved by the differential, and $d(e_1) = \sum x_i^2 = e_1^2 - 2e_2$. Thus $2e_2 \in I$. Unless the characteristic of $\Bbbk$ is $2$, this gives the desired
contradiction. \end{proof}

\begin{rem} Similarly, $d(e_k) = e_1 e_k - (k+1) e_{k+1}$ for all $k$. Ignoring the case of small primes, any ideal preserved by $I$ which contains $e_1$ will also contain $e_k$ for all $1 \le k \le n$. Thus $\Bbbk[x_1, \ldots, x_n]/I$ is finite-dimensional, as a quotient of the coinvariant ring. But its image inside $R$ is infinite dimensional,
containing at least a polynomial ring with $n-1$ generators. This is a contradiction. \end{rem}

The conclusion is that our realization with good differential need not be the standard one with the standard differential, but it (or the relevant part of it, the subring
generated by the $g_i$) is equivariantly isomorphic to the standard realization with the standard differential.

\begin{example} There is an $S_3$-invariant isomorphism $\Bbbk[x_1, x_2, x_3] \to \Bbbk[y_1, y_2, y_3]$ sending $x_1 \mapsto y_2 + y_3$, $x_2 \mapsto y_1 + y_3$, and $x_3 \mapsto
y_1 + y_2$. This intertwines the differential, when \[d(y_i) = y_i^2 + y_i y_j + y_i y_k - y_j y_k\] for all $\{i,j,k\} = \{1,2,3\}$. So, even when the realization is the
standard one, we can not and should not rule out the possibility that $R$ is equipped with a non-standard differential, because it might be a non-standard differential isomorphic
to a standard differential. \end{example}

The above proposition showed that the underlying polynomial rings of the standard realization and our good realization are isomorphic (outside of the possible characteristic $2$
exception), but we still need to show that the associated Hecke categories are isomorphic categories (via an isomorphism which intertwines the differential). This is slightly
subtle, because the isomorphism $\Bbbk[x_1, \ldots, x_n] \to R$ sending $x_i \mapsto g_i$ need not send simple roots to simple roots! Instead, \begin{equation} \alpha_i = x_i -
x_{i+1} \mapsto g_i - g_{i+1} = \kappa \alpha_i, \end{equation} and roots are rescaled by the scalar $\kappa$.

\begin{prop} Assume that the Hecke category is equipped with a good differential $d$. The homomorphism $\Bbbk[x_1, \ldots, x_n] \to R$ sending $x_i \mapsto g_i$ lifts to an
functor between the standard $p$-dg Hecke category and the version equipped with $d$. This functor rescales each $s$-colored enddot by $\kappa$, and each $s$-colored startdot by
$1$. Thus $\alpha_s \mapsto \kappa \alpha_s$, as noted above. The merging trivalent vertex is rescaled by $1$, and the splitting trivalent vertex is rescaled by $\kappa^{-1}$.
The 4- and 6-valent vertices are rescaled by $1$. This functor is an isomorphism, except in the case of characteristic $2$, when it is either an isomorphism or the kernel is generated by the polynomial $e_1$.
\end{prop}
	
\begin{proof} Let us assume that the map $\Bbbk[x_1, \ldots, x_n] \to R$ is an isomorphism (as it must be outside of characteristic $2$). In \cite{EHBraidToolkit} we classify all
autoequivalences of the Hecke category which fix the objects $B_s$ for each $s$. Any such automorphism is determined uniquely by how it rescales the start and end dots, and how
it affects the remainder of the polynomial ring (beyond the part spanned by roots), and these choices can be made arbitrarily. So, the functor defined above is an
autoequivalence. Checking that it intertwines the differential is straightforward. Modifying these results to account for the possible $e_1$ kernel in characteristic $2$ is straightforward. \end{proof}

\subsection{Relations to thick calculus} 

The thick calculus has a new object $B_{s,t}$ (drawn as purple) and several new morphisms: a trivalent vertex
\[ \ig{1}{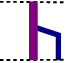} \]
$B_{s,t} B_s \to B_{s,t}$ together with variants thereof; splitters $B_{s,t} \to B_s B_t B_s$ and $B_{s,t} \to B_t B_s B_t$, and mergers $B_s B_t B_s \to B_{s,t}$ and $B_t B_s B_t \to B_{s,t}$, pictured below.
\[ \ig{1}{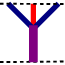} \qquad \ig{1}{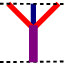} \qquad \ig{1}{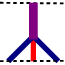} \qquad \ig{1}{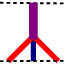} \] The composition of a splitter and a merger gives the $6$-valent vertex $B_s B_t B_s \to B_t B_s B_t$. The composition of a merger with a splitter gives the identity map of $B_{s,t}$.

Once again there is not very much new to say about differentials on the thick calculus. Here are the highlights.
\begin{itemize} \item Because all the direct summands of $B_s B_t B_s$ are actually objects in the thick diagrammatic category, one need not worry about the abstractions of \S\ref{subsec-mixedFc}, but can work directly with fantastic filtrations.
\item The differential on the trivalent vertex is forced to be the same as for an ordinary trivalent vertex (e.g. for the map $B_{s,t} B_s \to B_{s,t}$, put $-\barg_s$ in the middle on the bottom).
\item One can compute the general differential on the splitters and mergers. The computation is just as nasty and thorny as the one above for the general differential of the 6-valent vertices. Just as in \S\ref{ssec:thickdistant} there is one extra degree of freedom: the polynomial next to the splitter and the polynomial next to the merger can be arbitrary so long as they add up to the polynomial $f$ from \eqref{eq:dsix}. Also as in \S\ref{ssec:thickdistant}, this additional freedom disappears when restricting to good differentials. \end{itemize}

For posterity, here is the good differential on the thick calculus, when the orientation on the Dynkin diagram is $s \to t$.
\begin{subequations} \label{subeq:thickcalcdiff}
\begin{equation} 
d\left(~\ig{1.5}{thicktridown}~\right) = {
\labellist
\small\hair 2pt
 \pinlabel {$-\barg_s$} [ ] at 22 8
\endlabellist
\centering
\ig{1.5}{thicktridown}
} \qquad \qquad
d\left(~\ig{1.5}{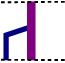}~\right) = {
\labellist
\small\hair 2pt
 \pinlabel {$-\barg_s$} [ ] at 9 8
\endlabellist
\centering
\ig{1.5}{thicktridownleft}
} \end{equation}
\begin{equation} 
d\left(~\ig{1.5}{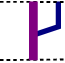}~\right) = {
\labellist
\small\hair 2pt
 \pinlabel {$-g_s$} [ ] at 22 22
\endlabellist
\centering
\ig{1.5}{thicktriup}
} \qquad \qquad
d\left(~\ig{1.5}{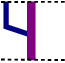}~\right) = {
\labellist
\small\hair 2pt
 \pinlabel {$-g_s$} [ ] at 9 22
\endlabellist
\centering
\ig{1.5}{thicktriupleft}
} \end{equation}

\begin{equation} 
d\left(~\ig{1.5}{thickadjsplit2}~\right) = 0 \qquad \qquad d\left(~\ig{1.5}{thickadjmerge}~\right) = 0 \end{equation}
\begin{equation} d\left(~\ig{1.5}{thickadjsplit}~\right) = \kappa \ig{1.5}{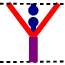} \qquad \qquad d\left(~\ig{1.5}{thickadjmerge2}~\right) = -\kappa \ig{1.5}{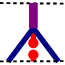} \end{equation}

\end{subequations}

\section{Three colors} \label{sec-three}

It remains to check the Zamolodchikov relations associated to finite rank 3 Coxeter subgroups. Since we are in simply laced type, there are only three possible rank 3 subgroups:
$A_1 \times A_1 \times A_1$, $A_1 \times A_2$, and $A_3$.

When the differential is good, the differential kills all $4$-valent vertices, and kills half of the 6-valent vertices. There is a version of each Zamolodchikov relation where the
differential kills both sides, and thus the relation is checked trivially! This is because the Zamolodchikov relations are equalities between two oriented paths in the reduced
expression graph. For the reverse-oriented good differential, all oriented paths are sent to zero by the differential. Rotating the relations by 180 degrees, we obtain an equivalent
relation which is an equality between reverse-oriented paths in the reduced expression graph, and these are sent to zero by the standardly oriented good differential.

\begin{rem} It is a good exercise for the reader learning diagrammatics to compute directly that the standardly oriented good differential preserves the oriented version of the
Zamolodchikov relation (the version that it does not just send to zero). \end{rem}

We have checked the $A_1 \times A_1 \times A_1$ and $A_1 \times A_2$ relations for an arbitrary differential, and they hold; the techniques required for this check have all been discussed above. We tried to check the $A_3$ relation for an arbitrary differential, but it was a surprisingly thorny computation, and we gave up.

\section{Does the differential have divided powers?} \label{sec-dp}
Yes, at least for the good differentials. We did not bother to check the general differential.

Consider the graded ring $T = \Z[x]$ with differential $d(x) = x^2$. Let $M$ be a free $T$-module of rank $1$, generated by the element $m$, and equip with with an edg-structure
(i.e. a degree $+2$ derivation) where $d_M(m) = ax \cdot m$ for some $a \in \Z$. We call this edg-module $M_a$. Then it is easy to compute that \begin{equation} d_M^{(k)} x^\ell
\cdot m = \binom{\ell+k+a-1}{k} x^{\ell + k} \cdot m \end{equation} for all $\ell \ge 0$. Thus the divided powers are defined over $\Z$.

\begin{rem} More generally, if $T = \Bbbk[x]$ for some field of characteristic $p$, and $M_a$ is defined as above for some $a \in \Bbbk$, then $d^p = 0$ on $M_a$ if and only if $a$ lives in the prime field $\F_p \subset \Bbbk$. \end{rem}

Now consider the diagrammatic Hecke category over $\Z$, and let $\theta$ be a generator. We ask whether $d^{(k)}(\theta)$ lives inside the $\Z$-form as well. Obviously this holds
if $d(\theta) = 0$, so we can consider only those differentials which do not kill $\theta$.

Suppose that $\theta$ is an $s$-colored enddot. Then $d(\theta) = g_s \theta$, and $d(g_s) = g_s^2$. Thus we can consider the sub-edg-algebra $T = \Z[g_s] \subset R$. The subset $T \cdot \theta$ is closed under the differential, and is isomorphic as an edg module over $T$ to $M_{1}$. In particular, $d^{(k)}(\theta)$ is well-defined over $\Z$. The same argument works for the $s$-colored startdot.

The analogous argument also works for the $s$-colored trivalent vertices, where $T$ acts by putting a polynomial in the appropriate region. This time $T \cdot \theta$ is isomorphic
to $M_{-1}$ instead.

Any good differential kills every 4-valent vertex, and one of the two 6-valent vertices. We need only check what happens to the other 6-valent vertex. We do the case when the orientation is $s \to t$.

Let $\alpha$, $\beta$, and $\gamma$ denote the following three diagrams.
\[ \alpha = \brokeXII, \qquad \beta = \brokeVI, \qquad \gamma = \brokeboth. \]

\begin{lem} The differential acts by the following formulae.
\begin{subequations}
\begin{equation} d(\alpha) = 2g_s \alpha - \kappa \gamma, \end{equation}
\begin{equation} d(\beta) = 2\barg_t \beta + \kappa \gamma, \end{equation}
\begin{equation} d(\gamma) = 2(g_s + \barg_t) \gamma. \end{equation}
\end{subequations}
Here, multiplication by a polynomial means putting that polynomial in the leftmost region (or rightmost region, it happens to be equal). 
\end{lem}

\begin{proof} The proof is straightforward. Let us derive the first equality, as the others are similar. Applying the differential to $\alpha$ we get a sum of three terms:
\begin{itemize} \item The differential applied to the broken strand on top (the pair of dots). This places $z_s$ in the top region.
\item Breaking the top strand again, with a plus sign and a factor of $\kappa$. This contributes $\kappa \alpha_s$ to the top region.
\item Breaking the bottom strand, with a minus sign and a factor of $\kappa$. This contributes $-\kappa \gamma$.
\end{itemize}
Then one observes that $z_s + \kappa \alpha_s = (g_s + \barg_s) + (g_s - \barg_s) = 2g_s$.  Since $g_s$ is $t$-fixed, it slides across the $t$-colored strand to the leftmost or rightmost region.
\end{proof}

\begin{lem} When the orientation is $s \to t$, the differential applied iteratively to the 6-valent vertex 
$$\phi \co B_sB_tB_s \to B_t B_sB_t$$
 is equal to
\begin{equation} d^k(\phi) = k! \kappa (g_s^{k-1} \alpha - \barg_t^{k-1} \beta - 
	\kappa (\sum_{a+b = k-2} g_s^a \barg_t^b) \gamma) \end{equation}
for any $k \ge 1$. Consequently, $d^{(k)}$ is well-defined integrally. \end{lem}

\begin{proof} We have $d(\phi) = \kappa(\alpha - \beta)$, and
\begin{equation} d^2(\phi) = \kappa(2 g_s \alpha - 2 \kappa \gamma  - 2 \barg_t \beta). \end{equation}
This proves the cases $k = 1, 2$. The inductive step is a simple exercise in the Leibniz rule. A helpful observation is that whenever $d(x) = x^2$ and $d(y) = y^2$ we have
\begin{equation} d(\sum_{a+b = m} x^a y^b) = m x^{m+1} + (m-1) xy \sum_{a'+b' = m-1} x^{a'} y^{b'} + m y^{m+1}. \end{equation} \end{proof}

Putting it all together, we have proven the following result.

\begin{thm} For any good differential $d$ on the diagrammatic Hecke category in simply-laced type, $d^{(k)}$ is defined integrally. \end{thm}

\section{A surprising example with $n=8$} \label{sec:n8example}

Let us work in the diagrammatic Hecke category for $S_8$ associated to the standard realization with polynomial ring $R = \Z[x_1, \ldots, x_8]$. Equip it with the standard differential, where $g_{s_i} = x_i$ and $\barg_{s_i} = x_{i+1}$ and $\kappa = 1$. We also extend this differential to the thicker calculus which includes objects $B_{sts}$ for $m_{st} = 3$, using \eqref{subeq:thickcalcdiff}.

Let $X$ denote the Bott-Samelson bimodule associated with the sequence
\[ \underline{w} = (s_3, s_2, s_1, s_5, s_4, s_3, s_2, s_6, s_5, s_4, s_3, s_7, s_6, s_5). \]
Let $y = s_3 s_2 s_3 s_5 s_6 s_5$. To describe the indecomposable object $B_y$ we can most easily use thick calculus, since $B_y \cong B_{s_3 s_2 s_3} B_{s_5 s_6 s_5}$.

Here is are degree $0$ map $p \co X \to B_y$ and $i \co B_y \to X$ which span their respective degree $0$ Hom spaces.
\begin{equation} p = \ig{1}{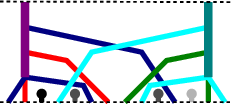} \qquad i = \ig{1}{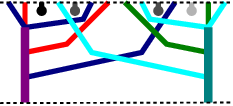} \end{equation}
To summarize these maps in words: the four (grayscale) strands with colors in $\{1,4,7\}$ get dotted off; the five (red and blue) strands with colors in $\{2,3\}$ get merged into $B_{s_3 s_2 s_3}$ (colored purple) with two thick trivalent vertices; the five (aqua and green) strands with colors in $\{5,6\}$ get merged into $B_{s_5 s_6 s_5}$ (colored teal) with two thick trivalent vertices. It is a very fun exercise to compute that
\begin{equation} pi = 2 \id_{B_y}. \end{equation}

A computation yields
\begin{equation} -d(p) = {{\color{olive}
\labellist
\small\hair 2pt
 \pinlabel {$x_6 - x_1$} [ ] at 54 19
 \pinlabel {$+1$} [ ] at 25 10
 \pinlabel {$+1$} [ ] at 99 6
\endlabellist
\centering
\ig{1.5}{n8proj}
}} \end{equation}
where again we use our abusive olive-colored sum notation and broken line notation from \S\ref{sec-abusivenotation}. To use words again, $-d(p)$ is the sum of three terms with the same underlying diagram as $p$: one which breaks the second strand colored $s_3$ with coefficient $+1$, one which breaks the last strand colored $s_6$ with coefficient $+1$, and one which places the polynomial $x_6 - x_1$ in the center. Another computation yields
\begin{equation} +d(i) = {{\color{olive}
\labellist
\small\hair 2pt
 \pinlabel {$x_8 - x_3$} [ ] at 54 34
 \pinlabel {$+1$} [ ] at 13 45
 \pinlabel {$+1$} [ ] at 86 40
\endlabellist
\centering
\ig{1.5}{n8inc}
}} \end{equation}
One can actually derive the second computation from the first, using various symmetries: rotation by 180 degrees, and the Dynkin diagram automorphism.

It is not immediately obvious that $d(p) \ne 0$, since it is not expressed as a linear combination of light leaves with polynomials on the right, but it is not too difficult to
justify. For example, the polynomial $x_1$ is fixed by $s_i$ for $i \ge 1$ and slides through the entire diagram; none of the other terms can possibly contribute a polynomial
involving $x_1$. Similarly, $d(i) \ne 0$ as can be seen from the $x_8$ term.

Finally, a computation yields
\begin{equation} d(p) i  = {{\color{olive}
\labellist
\small\hair 2pt
 \pinlabel {$x_1 + x_3 - x_6 - x_8$} [ ] at 40 28
 \pinlabel {$-1$} [ ] at 21 17
 \pinlabel {$-1$} [ ] at 61 17
\endlabellist
\centering
\ig{2}{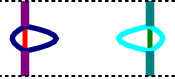}
}}. \end{equation}
Note that the split-merge appearing on the rightmost part of this picture is equal to the identity of $B_{s_5 s_6 s_5}$, but if the $s_5$-colored strand is broken (with coefficient $-1$), it yields a nonzero degree $+2$ endomorphism of $B_{s_5 s_6 s_5}$. Similarly with the leftmost part of the picture and the identity of $B_{s_3 s_2 s_3}$. 

As a consequence of this computation, $B_y$ is not a summand of $X$ in any dg-filtration, for any prime! This example shows that the conjectural $d$-canonical basis does not agree with the $p$-canonical basis. See \S\ref{ssec:intron8} for further discussion.

\addcontentsline{toc}{section}{References}


\bibliographystyle{alpha}
\bibliography{mastercopy}

\begin{thebibliography}{EMTW20}

\bibitem[BC18]{BeCo}
Anna Beliakova and Benjamin Cooper.
\newblock Steenrod structures on categorified quantum groups.
\newblock {\em Fund. Math.}, 241(2):179--207, 2018.
\newblock \href{https://arxiv.org/abs/1304.7152}{arXiv:1304.7152}.

\bibitem[EH]{EHBraidToolkit}
Ben Elias and Matthew Hogancamp.
\newblock Homotopy lifting and conjugation by {R}ouquier complexes.
\newblock In preparation.

\bibitem[EK10]{EKho}
Ben Elias and Mikhail Khovanov.
\newblock Diagrammatics for {S}oergel categories.
\newblock {\em Int. J. Math. Math. Sci.}, pages Art. ID 978635, 58, 2010.

\bibitem[Eli16a]{EThick}
Ben Elias.
\newblock Thicker {S}oergel calculus in type {$A$}.
\newblock {\em Proc. Lond. Math. Soc. (3)}, 112(5):924--978, 2016.

\bibitem[Eli16b]{ECathedral}
Ben Elias.
\newblock The two-color {S}oergel calculus.
\newblock {\em Compos. Math.}, 152(2):327--398, 2016.

\bibitem[EMTW20]{EMTW}
Ben Elias, Shotaro Makisumi, Ulrich Thiel, and Geordie Williamson.
\newblock {\em Introduction to {S}oergel bimodules}, volume~5 of {\em RSME
  Springer Series}.
\newblock Springer, 2020.

\bibitem[EQ16a]{EQpDGsmall}
Ben Elias and You Qi.
\newblock An approach to categorification of some small quantum groups {II}.
\newblock {\em Adv. Math.}, 288:81--151, 2016.

\bibitem[EQ16b]{EQpDGbig}
Ben Elias and You Qi.
\newblock A categorification of quantum {$\mathfrak{sl}(2)$} at prime roots of
  unity.
\newblock {\em Adv. Math.}, 299:863--930, 2016.

\bibitem[EQ21]{EQsl2}
Ben Elias and You Qi.
\newblock Actions of $\mathfrak{sl}_2$ on algebras appearing in
  categorification.
\newblock Preprint, 2021.
\newblock arXiv 2103.00048.

\bibitem[EW16]{EWGr4sb}
Ben Elias and Geordie Williamson.
\newblock Soergel calculus.
\newblock {\em Represent. Theory}, 20:295--374, 2016.
\newblock arXiv:1309.0865.

\bibitem[EW21]{EWrel}
Ben Elias and Geordie Williamson.
\newblock Relative hard {L}efschetz for {S}oergel bimodules.
\newblock {\em J. Eur. Math. Soc. (JEMS)}, 23(8):2549--2581, 2021.
\newblock arXiv:1607.03271.

\bibitem[EWS16]{EWSFrob}
Ben Elias, Geordie Williamson, and Noah Snyder.
\newblock On cubes of {F}robenius extensions.
\newblock In {\em Representation theory - current trends and perspectives},
  pages 171--186. European Mathematical Society, 2016.
\newblock arXiv:1308.5994.

\bibitem[Kho16]{KhoHopf}
Mikhail Khovanov.
\newblock Hopfological algebra and categorification at a root of unity: the
  first steps.
\newblock {\em J. Knot Theory Ramifications}, 25(3):1640006, 26, 2016.

\bibitem[Kit13]{Kitch}
Nitu Kitchloo.
\newblock Cohomology operations and the nil-{H}ecke ring.
\newblock 2013.
\newblock Preprint available at
  \href{http://www.math.jhu.edu/~nitu/papers/NH.pdf}{http://www.math.jhu.edu/~nitu/papers/NH.pdf}.

\bibitem[KL11]{KhoLau11}
Mikhail Khovanov and Aaron~D. Lauda.
\newblock A diagrammatic approach to categorification of quantum groups {II}.
\newblock {\em Trans. Amer. Math. Soc.}, 363(5):2685--2700, 2011.

\bibitem[Kle15]{KleAHW}
Alexander Kleshchev.
\newblock Affine highest weight categories and affine quasihereditary algebras.
\newblock {\em Proc. Lond. Math. Soc. (3)}, 110(4):841--882, 2015.
\newblock \href{https://arxiv.org/abs/1405.3328}{arXiv:1405.3328}.

\bibitem[KLMS12]{KLMS}
Mikhail Khovanov, Aaron~D. Lauda, Marco Mackaay, and Marko Sto{\v{s}}i{\'c}.
\newblock Extended graphical calculus for categorified quantum {${\rm sl}(2)$}.
\newblock {\em Mem. Amer. Math. Soc.}, 219(1029):vi+87, 2012.

\bibitem[KQ15]{KQ}
Mikhail Khovanov and You Qi.
\newblock An approach to categorification of some small quantum groups.
\newblock {\em Quantum Topol.}, 6(2):185--311, 2015.
\newblock \href{http://arxiv.org/abs/1208.0616}{arXiv:1208.0616}.

\bibitem[KR16]{KRWitt}
Mikhail Khovanov and Lev Rozansky.
\newblock Positive half of the {W}itt algebra acts on triply graded link
  homology.
\newblock {\em Quantum Topol.}, 7(4):737--795, 2016.
\newblock \href{https://arxiv.org/abs/1305.1642}{arXiv:1305.1642}.

\bibitem[Lau10]{LauSL2}
Aaron~D. Lauda.
\newblock A categorification of quantum {${\rm sl}(2)$}.
\newblock {\em Adv. Math.}, 225(6):3327--3424, 2010.

\bibitem[MSV13]{MSV}
Marco Mackaay, Marko Sto{\v{s}}i{\'c}, and Pedro Vaz.
\newblock A diagrammatic categorification of the {$q$}-{S}chur algebra.
\newblock {\em Quantum Topol.}, 4(1):1--75, 2013.

\bibitem[MV10]{MackaayVazFoams}
Marco Mackaay and Pedro Vaz.
\newblock The diagrammatic {S}oergel category and {${\rm sl}(N)$}-foams, for
  {$N\geq4$}.
\newblock {\em Int. J. Math. Math. Sci.}, pages Art. ID 468968, 20, 2010.

\bibitem[Qi14]{QiHopf}
You Qi.
\newblock Hopfological algebra.
\newblock {\em Compositio Mathematica}, 150(01):1--45, 2014.

\bibitem[QRSW22]{QRSW}
You Qi, Louis-Hadrien Robert, Joshua Sussan, and Emmanuel Wagner.
\newblock Symmetries of $\mathfrak{gl}_n$-foams.
\newblock 2022.
\newblock \href{https://arxiv.org/abs/2212.10106}{arXiv:2212.10106}.

\bibitem[QS18]{QiSussan3}
You Qi and Joshua Sussan.
\newblock p-{DG} cyclotomic nil{H}ecke algebras {II}.
\newblock Preprint, 2018.
\newblock arXiv 1811.04372.

\bibitem[Rou08]{Rouq2KM-pp}
Rapha{\"e}l Rouquier.
\newblock 2-{K}ac-{M}oody algebras.
\newblock Preprint, 2008.
\newblock arXiv:0812.5023.

\bibitem[Sch11]{SchPos}
Olaf~M. Schn\"{u}rer.
\newblock Perfect derived categories of positively graded {DG} algebras.
\newblock {\em Applied Categorical Structures}, 19(5):757--782, 2011.
\newblock \href{http://arxiv.org/abs/0809.4782}{arXiv:0809.4782}.

\bibitem[Shu14]{Shu}
Alexander~N. Shumakovitch.
\newblock Torsion of {K}hovanov homology.
\newblock {\em Fund. Math.}, 225(1):343--364, 2014.

\bibitem[Soe90]{Soer90}
Wolfgang Soergel.
\newblock Kategorie {$\mathcal{O}$}, perverse {G}arben und {M}oduln \"uber den
  {K}oinvarianten zur {W}eylgruppe.
\newblock {\em J. Amer. Math. Soc.}, 3(2):421--445, 1990.

\bibitem[Ste18]{StephensSL3}
Andrew Stephens.
\newblock {\em A categorification of quantum {$\mathfrak{sl}(3)$} at a prime
  root of unity}.
\newblock PhD thesis, University of Oregon, September 2018.

\bibitem[Sto11]{StosicSL3}
Marko Stosic.
\newblock Indecomposable 1-morphisms of $\dot{U}^+_3$ and the canonical basis
  of ${U}_q^+(sl_3)$.
\newblock Preprint, 2011.
\newblock arXiv:1105.4458.

\bibitem[Vaz10]{VazFoams}
Pedro Vaz.
\newblock The diagrammatic {S}oergel category and {$\rm sl(2)$} and {$\rm
  sl(3)$} foams.
\newblock {\em Int. J. Math. Math. Sci.}, pages Art. ID 612360, 23, 2010.

\bibitem[Wan21]{JWang}
Joshua Wang.
\newblock On sl({N}) link homology with mod {N} coefficients.
\newblock 2021.
\newblock \href{https://arxiv.org/abs/2111.02287}{arXiv:2111.02287}.

\bibitem[Wil17]{WillCounter}
Geordie Williamson.
\newblock Schubert calculus and torsion explosion.
\newblock {\em J. Amer. Math. Soc.}, 30(4):1023--1046, 2017.
\newblock With a joint appendix with Alex Kontorovich and Peter J. McNamara.

\end{thebibliography}

%

\vspace{0.1in}

\noindent B.~E.:  { \sl \small Department of Mathematics, University of Oregon, Eugene, OR 97403, USA}\newline \noindent {\tt \small email: belias@uoregon.edu}

\vspace{0.1in}

\noindent Y.~Q.: { \sl \small Department of Mathematics, University of Virginia, Charlottesville, VA 22901, USA} \newline \noindent {\tt \small email: yq2dw@virginia.edu}

%
\end{document}